\documentclass[a4paper,11pt]{article}

\usepackage{a4wide}
\usepackage{amsfonts}
\usepackage{graphics}
\usepackage{amsmath, amsthm}  
\usepackage{amssymb,bbm} 
\usepackage{enumerate}
\usepackage{amsthm}
\usepackage{hyperref}
\usepackage{cite}
\usepackage[all,cmtip]{xy}

\hypersetup{
	bookmarks=true,         
	unicode=false,          
	pdftoolbar=true,        
	pdfmenubar=true,        
	pdffitwindow=true,     
	pdfstartview={FitH},    
	pdftitle={Nilpotent groups and identities of automorphisms},    
	pdfauthor={Wolfgang Alexander Moens},     
	pdfsubject={Survey},   
	pdfkeywords={Groups, Automorphisms, Identities, Lie algebras, Gradings}, 
	pdfnewwindow=true,      
	colorlinks=true,       
	linkcolor=gray,          
	citecolor=red,        
	filecolor=magenta,      
	urlcolor=gray           
}

\usepackage{tikz}
\usepackage{epstopdf}

\newcommand{\C}{\mathbb{C}}\newcommand{\id}{\mathbbm{1}}\newcommand{\N}{\mathbb{N}}
\newcommand{\Q}{\mathbb{Q}}\newcommand{\Z}{\mathbb{Z}}\newcommand{\F}{\mathbb{F}}

\newcommand{\map}[3]{ #1 : #2 \longrightarrow #3 }

\newcommand{\mapl}[5]{ #1 : #2 \longrightarrow #3 : #4 \longmapsto #5 }

\newcommand{\tcn}[1]{ \operatorname{Cong}(#1) }
\newcommand{\discr}[1]{ \operatorname{Discr}_\ast(#1) }
\newcommand{\prodant}[1]{ \operatorname{Prod}(#1) }

\newtheorem{theorem}{Theorem}[subsection]

\newtheorem{proposition}[theorem]{Proposition}

\newtheorem{corollary}[theorem]{Corollary}
\newtheorem*{corollary*}{Corollary}

\newtheorem{lemma}[theorem]{Lemma}

\newtheorem*{theorem*}{Theorem}

\newtheorem*{proposition*}{Proposition}
\newtheorem*{lemma*}{Lemma}
\newtheorem*{problem*}{Problem}

\newtheorem{metaproblem}[theorem]{Meta-Problem}
\newtheorem*{observation*}{Observation}

\theoremstyle{definition} \newtheorem{remark}[theorem]{Remark}
\theoremstyle{definition} \newtheorem{example}[theorem]{Example}
\theoremstyle{definition}\newtheorem{definition}[theorem]{Definition}
\theoremstyle{definition}

\usepackage{authblk}
\title{The structure of finite groups with \\a f.p.f. automorphism satisfying an identity \thanks{This work was supported by the Austrian Science Fund (FWF) grants: $J-3371-N25$ (\emph{``Representations and gradings of solvable Lie algebras''}) and $P 30842-N35$ (\emph{``Infinitesimal Lie rings: gradings and obstructions''}).}}

\author{Wolfgang Alexander Moens
}
\affil{\textsc{Faculty of Mathematics, University of Vienna, Oskar-Morgenstern-Platz 1, 1090 Vienna, Austria}\\ \emph{E-Mail:} Wolfgang.Moens@univie.ac.at\\ \emph{Tel:}  +43-1-4277-50468}
\date{\today}

\begin{document}
	
	\maketitle

\abstract{
	We generalize the positive solution of the Frobenius conjecture and refinements thereof by studying the structure of groups that admit a fix-point-free automorphism satisfying an identity. \\
	
	We show, in particular, that for every polynomial $r(t) = a_0 + a_1 \cdot t + \cdots + a_d \cdot t^d \in \Z[t]$ that is irreducible over $\Q$, there exist (explicit) invariants $a,b,c \in \N$ with the following property. Consider a finite group with a fix-point-free automorphism $\map{\alpha}{G}{G}$ and suppose that for all $x \in G$ we have the equality $x^{a_0} \cdot \alpha(x^{a_1}) \cdot \alpha^2(x^{a_2})\cdots \alpha^d(x^{a_d}) = 1_G.$ Then $G$ is solvable and of the form $A \cdot (B \rtimes (C \times D))$, where $A$ is an $a$-group, $B$ is a $b$-group, $C$ is a nilpotent $c$-group, and $D$ is a nilpotent group of class at most $d^{2^d}$. 	Here, a group $H$ is said to be an $a$-group (resp. $b$-group or $c$-group) if the order of every $h \in H$ divides some natural power of $a$ (resp. $b$ or $c$).
}

	\tableofcontents

\section{Introduction} \label{SectionIntro}

\subsection{Motivation} \label{SubsectionMotivation}

In this text, we study the structure of finite groups with a fix-point-free automorphism satisfying an ``identity.'' Before clarifying what we mean by ``identity,'' we recall three classic results on finite groups with a fix-point-free automorphism.

\begin{theorem}[Rowley \cite{Rowley}] \label{TheoremA} If a finite group $G$ admits a fix-point-free automorphism, then $G$ is solvable. \end{theorem}

We recall that an automorphism is called \emph{fix-point-free} (or \emph{regular}) if it fixes only the trivial element of the group. Theorem \ref{TheoremA} has a long history, going back to at least Gorenstein---Herstein \cite{GorensteinHerstein}, and it was finally confirmed by means of the classification of the finite simple groups. We refer to Rowley's paper for a particularly short proof. 
By considering a special case, we can hope to obtain a stronger conclusion.

\begin{theorem}[J. Thompson \cite{ThompsonFixPointFreeAutomorphisms}] \label{TheoremB} If a finite group $G$ admits a fix-point-free automorphism of prime order, then $G$ is nilpotent. \end{theorem}

Such automorphisms naturally appear in the study of groups acting simply-transitively on finite sets, and theorem \ref{TheoremB} gives a positive answer to (what is generally known as) the Frobenius conjecture. Thompson's proof used the celebrated $p$-complement theorem \cite{ThompsonNormalComplements} and an earlier result of Witt and Higman, but it did not require the classification of the finite simple groups. 
A follow-up result is:

\begin{theorem}[Higman \cite{HigmanGroupsAndRings}; Kreknin---Kostrikin \cite{KrekninAutomorphismFinitePeriod,KrekninKostrikinRegularAutomorphisms}] \label{TheoremC} If a nilpotent group $G$ admits a fix-point-free automorphism of prime order $p$, then its class satisfies $\operatorname{c}(G) \leq (p-1)^{2^{(p-1)}}$. \end{theorem}

In \cite{HigmanGroupsAndRings}, Higman proved that there exists some upper bound for $\operatorname{c}(G)$ that depends only on $p$. Kreknin and Kostrikin later proved in \cite{KrekninAutomorphismFinitePeriod,KrekninKostrikinRegularAutomorphisms} an upper bound of $(p-1)^{2^{(p-1)}}$. But it is conjectured that the minimal upper bound $h(p)$ on $\operatorname{c}(G)$ satisfies $h(p) = \lceil (p^2-1) / 4 \rceil$. This problem has been referred to as the Higman conjecture, and it still open for primes $p \geq 11$. 
We next make some elementary observations. 

\begin{remark} \label{RmkFiniteOrderToCyclo}
	We suppose, as in the above theorems, that the group $G$ is finite or nilpotent, and that $\map{\alpha}{G}{G}$ is a fix-point-free automorphism of finite order $n$. Then the transformation $1+\alpha + \alpha^2 + \cdots + \alpha^{n-1}$ of $G$, defined by $x \mapsto x \cdot \alpha(x) \cdot \alpha^2(x) \cdots \alpha^{n-1}(x),$ vanishes identically. If, moreover, $n$ is a prime, then the group $G$ will have no $n$-torsion. 
\end{remark}

In view of these impressive results, we propose the following family of problems.

\begin{metaproblem} \label{MetaProblem} We are given a polynomial $r(t) := a_0 + a_1 \cdot t + \cdots + a_d \cdot t^d $ with integer coefficients, and we are told that some finite group $G$ admits a \emph{fix-point-free} automorphism $\map{\alpha}{G}{G}$ such that the map $G \longrightarrow G$, defined by $$ x \mapsto x^{a_0} \cdot \alpha( x^{a_1} ) \cdot \alpha^2(x^{a_2}) \cdots \alpha^d(x^{a_d}),$$ vanishes identically. Prove that the group $G$ has a ``good structure'' modulo ``bad torsion.'' \end{metaproblem}

The main objective of this text is to refine theorem \ref{TheoremA} and to generalize theorems \ref{TheoremB} and \ref{TheoremC} in the spirit of Meta-Problem \ref{MetaProblem}. We will do this in theorem \ref{MainCorollaryWithClassification}, theorem \ref{MainTheoremA}, and theorem \ref{MainTheoremB} respectively. This will then allow us to generalize other classic results of Hughes---Thompson, Kegel and E. Khukhro on the structure of finite groups with a $p$-splitting automorphism. We will do this in corollary \ref{TheoremTHHUTHKEG} and corollary \ref{TheoremHKKK}. \\

In order to state our results in full generality, we will have to introduce some notation and terminology. But we can already formulate a corollary that captures the spirit of our main results.

\begin{corollary}[Irreducible case] \label{CorollIrreducibleConstants}
	Let $r(t) = a_0 + a_1 \cdot t + \cdots + a_d \cdot t^d \in \Z[t]$ be a polynomial that is \emph{irreducible} over $\Q$. Then there exist constants $a,b,c \in \N$ with the following property. Consider a finite group $G$ with a fix-point-free automorphism $\map{\alpha}{G}{G}$ and suppose that for all $x \in G$ we have the equality $$x^{a_0} \cdot \alpha(x^{a_1}) \cdot \alpha^2(x^{a_2})\cdots \alpha^d(x^{a_d}) = 1_G.$$ Then $G$ is solvable and of the form $A \cdot (B \rtimes (C \times D))$, where $A$ is an $a$-group, $B$ is a $b$-group, $C$ is a nilpotent $c$-group, and $D$ is a nilpotent group of class at most $d^{2^d}$. 
	In particular: 
	\begin{itemize}
		\item If $G$ has no $(a \cdot b)$-torsion, then $G$ is nilpotent.
		\item If $G$ has no $(a \cdot b \cdot c)$-torsion, then $G$ is nilpotent of class at most $d^{2^d}$.
	\end{itemize}
\end{corollary}

Here, we recall terminology from \cite{Baer}. Let $n \in \Z$. A group $H$ is said to be an \emph{$n$-group} if the order of every element of $H$ divides some natural power of $n$. We say that a group $H$ \emph{has $n$-torsion} if some $h \in H \setminus \{1_H\}$ satisfies $h^n = 1_H$. Otherwise, we say that $H$ has no $n$-torsion. This naturally generalizes the notion of $p$-groups and $p$-torsion (for $p$ a prime).

\subsection{Main results}

\subsubsection{Invariants of good identities}

\begin{definition}[Identities]
	We consider a group $(G,\cdot)$, together with an endomorphism $\map{\gamma}{G}{G}$, and a polynomial $r(t) := a_0 + a_1 \cdot t + \cdots + a_d \cdot t^d \in \Z[t]$. We say that $r(t)$ is a \emph{monotone identity} of $\gamma$ if and only if the map $\mapl{r(\gamma)}{G}{G}{x}{x^{r(\gamma)} := x^{a_0} \cdot \gamma(x^{a_1}) \cdots \gamma^d(x^{a_d})}$	vanishes identically. In this case, we will simply write $r(\gamma) = 1_G$. More generally, we say that $r(t)$ is an \emph{identity} of $\gamma$ if and only if there exists a decomposition $r(t) = r_1(t) + \cdots + r_k(t)$ of $r(t)$ into polynomials $r_1(t),\ldots, r_k(t) \in \Z[t]$, such that the map $\mapl{r_1(\gamma) \cdots r_k(\gamma)}{G}{G}{x}{x^{r_1(\gamma)} \cdots x^{r_k(\gamma)}}$  vanishes identically. We will abbreviate this to $r_1(\gamma) \cdots r_k(\gamma) = 1_G$.
\end{definition}

\begin{definition}[Good polynomials]
	A polynomial $r(t) \in \Z[t]$ is said to be \emph{bad} if and only if  there exists some $u \in \N$ and some $s(t) \in \Z[t] \setminus \Z$ such that $s(t^{u+1}) \mid r(0) \cdot r(1) \cdot r(t)$ in the ring $\Z[t]$. Otherwise, we say that $r(t)$ is \emph{good}.
\end{definition}

For each polynomial $r(t) \in \Z[t]$, we will eventually introduce integer invariants $ \tcn{r(t)}$, $\discr{r(t)}$, and $\prodant{r(t)}$. We postpone their precise definitions to definition \ref{DefCongnr} and \ref{DefDiscProd}. For now, it suffices to know that good polynomials have non-zero invariants:

\begin{proposition} \label{PropositionGoodPolynomialsVSInvariants}
	Consider a good polynomial $r(t) \in \Z[t]$. Then $ r(1) \cdot \tcn{r(t)} \cdot \discr{r(t)} \cdot \prodant{r(t)} \neq 0.$
\end{proposition}

\subsubsection{Structure} 

\begin{theorem} \label{MainTheoremA}
	Let $G$ be a finite group admitting a fix-point-free automorphism  with a good identity $r(t)$. Then either $G$ has $\tcn{r(t)}$-torsion or $G$ is nilpotent.
\end{theorem}

We will prove the theorem by extending classic techniques of Higman \cite{HigmanGroupsAndRings,HigmanLieMethods} and J. Thompson \cite{ThompsonFixPointFreeAutomorphisms} in the context of the Frobenius conjecture.

\begin{theorem} \label{MainTheoremB}
	Let $G$ be a nilpotent group admitting an endomorphism with a good identity $r(t)$, say of degree $d$. Then either $G$ has $(\discr{r(t)} \cdot \prodant{r(t)})$-torsion or $G$ is nilpotent of class $\operatorname{c}(G) \leq {d}^{2^{d}}$.
\end{theorem}

We will prove the theorem with various techniques from the theory of Lie rings. We will prove, in particular, the following auxiliary result.

\begin{theorem} \label{IntroTheoremLie} Consider a Lie ring $L$, together with an endomorphism $\map{\gamma}{L}{L}$ of the Lie ring, and a good polynomial $r(t) \in \Z[t]$ satisfying $r(\gamma) = 0_L$. Suppose $(L,+)$ has no $(\discr{r(t)} \cdot \prodant{r(t)})$-torsion. Then $G$ is nilpotent of class $\operatorname{c}(G) \leq d^{2^d}$, where $d = \deg(r(t))$. \end{theorem}

By combining theorems \ref{MainTheoremA} and \ref{MainTheoremB}, we immediately obtain:

\begin{corollary}[Bad torsion or bounded nilpotency] \label{MainCorollaryNoClassification} \label{CorollCompactMainThm}
	Let $G$ be a finite group admitting a fix-point-free autorphism with a good identity $r(t)$, say of degree $d$. Then either $G$ has $\tcn{r(t)}$-torsion or $G$ is of the form $G \cong C \times D,$ where $C$ is a nilpotent $(\discr{r(t)} \cdot \prodant{r(t)})$-group and $D$ is a nilpotent group of class at most $d^{2^{d}}$.
\end{corollary}

None of the above results depends on the classification of the finite simple groups. But, by using the classification, we can obtain a stronger result.

\begin{theorem}[Bounded nilpotency modulo bad torsion] \label{MainCorollaryWithClassification} Let $G$ be a finite group admitting a fix-point-free automorphism with a good identity $r(t)$. Then $G$ is solvable and of the form $G \cong B \rtimes (C \times D),$ where $B$ is a solvable $\tcn{r(t)}$-group, $C$ is a nilpotent $(\discr{r(t)} \cdot \prodant{r(t)})$-group, and $D$ is a nilpotent group of class at most $d^{2^d}$. \end{theorem} 

With a little more work, we can extend the theorem to identities that are irreducible over $\Q$: see corollary \ref{CorollIrreducibleConstants}. Since not all irreducible polynomials are good, one could hope to further extend theorem \ref{MainCorollaryWithClassification} \emph{all} polynomials. But this is not possible. We will comment on this in the closing remarks.

\subsubsection{Existence} 

Meta-Problem \ref{MetaProblem} raises two obvious questions: which fix-point-free automorphisms of a group with ``good structure'' have a (monic) identity, and conversely, which (monic) polynomials are an identity of a fix-point-free automorphism of a group with ``good structure''? It is easy to answer these questions if we interpret ``good structure'' to mean ``finitely-generated and nilpotent.'' In fact, a straightforward generalization of the theorem of Cayley---Hamilton shows:

\begin{proposition}[Existence of $r(t)$] \label{PropIntroExistenceIdentities}
	Consider a finitely-generated, nilpotent group $G$. Then every endomorphism $\map{\gamma}{G}{G}$ has a monic identity $r(t)$ of degree  $ \operatorname{deg}(r(t)) \leq \operatorname{Hirsch}(G) + |\operatorname{Tors}(G)|$.
\end{proposition}

Here, $\operatorname{Hirsch}(G)$ is the Hirsch-length of $G$ and $\operatorname{Tors}(G)$ is the torsion-subgroup of $G$. Lie-theoretic techniques of Higman and the classic Mal$'$cev-correspondence give a quantitative answer to the second question:

\begin{proposition}[Existence of $G$ and $\gamma$] \label{PropLongProgLieAlgGroup} Consider a monic polynomial $r(t) \in \Z[t] \setminus \{ 1 \}$. Let $\lambda$ and $\mu$ be roots of $r(t)$ in $\overline{\Q}$, and let $k$ be any natural number satisfying $r(\lambda) = r(\lambda \cdot \mu) = r(\lambda \cdot \mu^2) = \cdots = r(\lambda \cdot \mu^{k-1}) = 0.$ 
	\begin{enumerate}[(i.)]
		\item Then there is a finitely-generated, torsion-free, $k$-step nilpotent group $G$ admitting an endomorphism $\map{\gamma}{G}{G}$ that has $r(t)^k$ as an identity.
		\item For every prime $p$, there is a finite, $k$-step nilpotent $p$-group $G$ admitting an endomorphism $\map{\gamma}{G}{G}$ that has $r(t)^k$ as an identity. If $r(0) \cdot r(1) \not \equiv 0  \mod p$, then such a $\gamma$ must be a fix-point-free automorphism.
	\end{enumerate}
\end{proposition}

Note that we can always make the minimal choice $k := 1$. And, since natural powers $r(t)^k$ of good polynomials $r(t)$ are again good, we may conclude that \emph{every good polynomial of positive degree is an identity of some fix-point-free automorphism of some finite, nilpotent group of non-trivial class and of almost-arbitrary torsion}.

\subsection{Outline}

In {section \ref{SectionExistenceOfIdentities}}, we prove the existence of identities and endomorphisms in the context of nilpotent groups. In section \ref{SectionInvars}, we define the invariants $\tcn{r(t)}, \discr{r(t)},\prodant{r(t)}$ and we study their basic properties. We then illustrate these results in section \ref{SectionExamplesForLater} by means of examples. Next, we prove our auxiliary results about Lie rings in section \ref{SectionLieRingsAux}. In section \ref{SectionStructureForGroups}, we prove our main results and corollaries. We then specialize our results to linear, cyclotomic, and Anosov polynomials in section \ref{SectionExamples}. We conclude with some remarks.

\section{Existence} \label{SectionExistenceOfIdentities}

\subsection{Preliminaries}

We begin with a simple observation.

\begin{lemma}[Composition of polynomial maps] \label{LemmaComposition} Let $m,k_1,\ldots,k_m \in \N$ and let $$u_{(1,1)}(t),\ldots,u_{(1,k_1)}(t),\ldots,u_{(m,1)}(t),\ldots,u_{(m,k_m)}(t)$$ be polynomials with integer coefficients. Then there exists an $n \in \N$ and polynomials $s_1(t),\ldots,s_n(t) \in \Z[t]$ such that $$\prod_{1 \leq j \leq m} \sum_{1 \leq i \leq k_j} u_{(j,i)}(t) = s_1(t) + \cdots + s_n(t)$$ and such that for all group endomorphisms $\map{\gamma}{G}{G}$, we have the equality of maps \begin{equation} \label{EqDecoMap}  (u_{(m,1)}(\gamma) \cdots u_{(m,k_m)}(\gamma)) \circ \cdots \circ (u_{(1,1)}(\gamma) \cdots u_{(1,k_1)}(\gamma)) = s_1(\gamma) \cdots s_n(\gamma) .\end{equation} In particular: if \eqref{EqDecoMap} vanishes identically, then $ \prod_j \sum_i u_{(j,i)}(t)$ is an identity of $\gamma$. \end{lemma}

The proof is a straight-forward induction on $m \in \N$, and we leave it to the reader. As an immediate consequence, we obtain:

\begin{proposition}[Ideal of identities] \label{PropIdentities} 
	Let $G$ be a group and let $\map{\gamma}{G}{G}$ be an endomorphism of $G$. Then the identities of $\gamma$ form an ideal of $\Z[t]$. \end{proposition}

\subsection{Proof of proposition \ref{PropIntroExistenceIdentities}} \label{SubSecConstrIdentity}
\begin{lemma} \label{LemmaComposingIdentities}
	Consider a group $G$ with an endomorphism $\map{\gamma}{G}{G}$. Suppose that $G$ admits a subnormal series $G = G_1 \trianglerighteq G_2 \trianglerighteq \cdots \trianglerighteq G_l \trianglerighteq G_{l+1} = \{ 1_G \}$ of $\gamma$-invariant subgroups and identities $r_1(t),\ldots, r_l(t) \in \Z[t]$ of the induced endomorphisms $\map{\gamma_{G_i / G_{i+1}}}{G_i / G_{i+1}}{G_i / G_{i+1}}$ on the factors $G_i / G_{i+1}$. Then $r_1(t) \cdots r_l(t)$ is an identity of $\gamma$.
\end{lemma}

\begin{proof}
	For each $r_j(t)$, there exists a $k_j \in \N$ and polynomials $u_{(j,1)}(t),\ldots,u_{(j,k_j)} \in \Z[t]$ such that $\sum_i u_{(j,i)} = r_j(t)$ and such that $u_{j,1}(\gamma_{G_j / G_{j+1}}) \cdots u_{j,k_j}(\gamma_{G_j / G_{j+1}}) = 1_{G_j / G_{j+1}}$. So the map $u_{(j,1)}(\gamma) \cdots u_{(j,k_j)}(\gamma)$ sends $G_j$ into $G_{j+1}$. The composition of these maps therefore vanishes on all of $G$. 	Lemma~\ref{LemmaComposition} now implies that $\prod_j \sum_i u_{(j,i)}(t) = r_1(t) \cdots r_l(t)$ is an identity of $\gamma$.
\end{proof}

\begin{proposition}[Cayley---Hamilton] \label{TheoremCH}
	Consider a group $G$ with an endomorphism $\map{\gamma}{G}{G}$. If $G$ admits a subnormal series $G = G_1 \trianglerighteq G_2 \trianglerighteq \cdots \trianglerighteq G_l \trianglerighteq G_{l+1} = \{ 1_G \}$ of $\gamma$-invariant subgroups such that every factor $G_i / G_{i+1}$ is free-abelian of finite rank or elementary-abelian of finite rank, then the endomorphism $\gamma$ has a monic identity $\chi(t)$ of degree $\operatorname{d}(G_1 / G_{2}) + \cdots + \operatorname{d}(G_l / G_{l+1})$.
\end{proposition}

\begin{proof}
	If $G_i / G_{i+1}$ is free-abelian, then we may compute the characteristic polynomial $\chi_i(t) = \det(t \cdot \id_{G_i / G_{i+1}} - \gamma_{G_i / G_{i+1}}) \in \Z[t]$ of the induced endomorphism $\gamma_{G_i / G_{i+1}}$. Else, the factor $G_i / G_{i+1}$ is elementary-abelian, say isomorphic to $(\Z_p^k,+)$, so that we may compute the characteristic polynomial $\overline{\chi}_i(t) \in \F_p[t]$ of $\gamma_{G_i / G_{i+1}}$ over the field $\F_p$. There then exists a monic polynomial $\chi_i(t) \in \Z[t]$ of the same degree $k$ as $\overline{\chi}_i(t)$ such that $\chi_i(t) \operatorname{mod} p = \overline{\chi}_i(t)$. According to the (classic) theorem of Cayley---Hamilton and lemma \ref{LemmaComposingIdentities}, the product $\chi(t) := \chi_1(t) \cdot \chi_2(t) \cdots \chi_{l}(t) \in \Z[t],$ is an identity of $\gamma$, and $\chi(t)$ clearly has degree $\operatorname{d}(G_1 / G_2) + \cdots + \operatorname{d}(G_l / G_{l+1})$. Since each $\chi_i(t)$ is monic, so is $\chi(t)$.
\end{proof}

If all the factors $G_i / G_{i+1}$ in proposition \ref{TheoremCH} are free-abelian of finite rank, then the polynomial $\chi(t)$ is uniquely determined by this construction, so that we may refer to $\chi(t)$ as the \emph{characteristic polynomial of the endomorphism with respect to the series} $(G_i)_i$.

\begin{corollary}
	\label{CorFinSolv}
	Consider a finite, solvable group $G$. Then every endomorphism $\map{\gamma}{G}{G}$ of $S$ admits a monic identity $r(t)$ of degree $\operatorname{deg}(r(t)) \leq |G|$.
\end{corollary}

\begin{proof}
	Since $G$ is finite and solvable, it admits a fully-invariant series $(G_i)_i$ with elementary-abelian factors, so that we may apply proposition \ref{TheoremCH}. It then suffices to observe that $\sum_i \operatorname{d}(G_i/G_{i+1}) \leq |G|$.
\end{proof}

\begin{proof}(Proposition \ref{PropIntroExistenceIdentities})
	Let us consider the (well-defined) torsion subgroup $T := \operatorname{Tors}(G)$ of $G$, the fully-invariant series $G \geq T \geq \{ 1_G \}$ of $G$, and the induced endomorphisms on the factors. The subgroup $T$ is finite and nilpotent so that, it has a fully-invariant series with elementary-abelian factors. So proposition \ref{TheoremCH} provides a monic identity $\chi_T(t)$ of degree at most $|T|$ for the naturally induced automorphism $\map{\gamma_{T}}{T}{T}$. Let us now consider the induced endomorphism $\map{\gamma_{G/T}}{Q}{Q}$ on the factor $Q := G/T$. This $Q$ is finitely-generated, torsion-free, and nilpotent. Moreover, its Hirsch-length coincides with that of $G$. For each term $\Gamma_i(Q)$ of the lower central series of $Q$, we consider the (fully-invariant) isolator subgroup $Q_i := \Gamma_i^\ast(Q)$ in $Q$. Then $(Q_i)_i$ satisfies the conditions of proposition \ref{TheoremCH}, so that we obtain a monic identity $\chi_{G/T}(t)$ of $\gamma_{G/T}$ of degree equal to the Hirsch-length of $Q$. Lemma \ref{LemmaComposingIdentities} now shows that the product $r(t) := \chi_{G/T}(t) \cdot \chi_T(t)$ is a monic identity of $\gamma$ with the correct degree. Similarly, lemma \ref{LemmaComposingIdentities} shows that $\gamma$ admits the (not necessarily monic) identity $s(t) := |T| \cdot \chi_{G/T}(t)$ satisfying $\operatorname{deg}(s(t)) \leq \operatorname{Hirsch}(G)$.
\end{proof}

We conclude with a remark that will be useful later on.

\begin{remark} \label{RemarkChar0} Assume that the $\gamma$ of theorem \ref{TheoremCH} is an automorphism. If $\chi(0) = \pm 1$, then every $\gamma$-invariant subgroup $M$ of $G$ is also $\langle \gamma \rangle$-invariant, so that the induced map $\map{\gamma_M}{M}{M}$ is an automorphism of $G$. \end{remark}

\begin{proof}
	It suffices to show that $\gamma^{-1}(M) \subseteq M$. One can use induction on $l \in \N$ to show that for every $x \in M$, we have $(\chi_l(\gamma) \circ \cdots \circ \chi_1(\gamma))(x) \in x^{\chi(0)} \cdot \gamma(M).$ Then theorem \ref{TheoremCH} gives $1_G \in x^{\chi(0)} \cdot \gamma(M)$. So, if $\chi(0) = \pm 1$, then $\gamma^{-1}(x) \in M$.
\end{proof}

\subsection{Proof of proposition \ref{PropLongProgLieAlgGroup}} \label{SubSectionConstructionAuto}

We begin by proving the corresponding statement for Lie algebras.

\begin{proposition} \label{PropConstrEndoTorsionFreeLie} Consider a monic polynomial $r(t) \in \Z[t] \setminus \{ 1 \}$ with $r(0) \neq 0$. Let $\lambda$ and $\mu$ be roots of $r(t)$ in $\overline{\Q}$, and let $k$ be any natural number satisfying $r(\lambda \cdot \mu) = \cdots = r(\lambda \cdot \mu^{k-1}) = 0.$ Then there is a finitely-generated, $k$-step nilpotent Lie algebra $L$ over the rational numbers with an automorphism $\map{\overline{\alpha}}{L}{L}$ such that $r(\overline{\alpha}) = 0_L$. \end{proposition}

\begin{proof}
	According to proposition \ref{PropIdentities}, the identities of an endomorphism form an ideal of $\Z[t]$. So we may assume that $r(t)$ is square-free and that $\operatorname{Discr}{r(t)} \neq 0$. If $k = 1$, then we simply consider the companion operator $\map{\gamma}{\Q^{\deg(r(t))}}{\Q^{\deg(r(t))}}$ of $r(t)$ on the abelian group $(\Q^{\deg(r(t))},+)$. It is well-known that this endomorphism $\gamma$ satisfies $r(\gamma) = 0_{\Q^{\deg(r(t))}}$. \newline
	
	So we may further assume that $k \geq 2$. We let $F$ be the free $k$-step nilpotent Lie algebra (over the rational numbers) on the generators $x_{1,1},\ldots,x_{1,d},x_{2,1},\ldots,x_{2,d}$. Let $C \in \operatorname{GL}_d(\Q)$ be the companion operator of $r(t)$ and let $A$ be the direct sum $C \oplus C \in \operatorname{GL}_{2 d}(\Q) \cap  \operatorname{Mat}_{2d,2d}(\Z)$. This $A$  defines a linear transformation of the $\Q$-span of the generators of $F$ (in the obvious way) and $A$ extends (in a unique way) to an automorphism $\map{\alpha}{F}{F}$ of the Lie algebra $F$. 
	We now consider the ideal $I$ of $F$ that is generated by the subset $(r(\alpha))(F)$ of $F$. This ideal is $\langle \alpha \rangle$-invariant, so that we may consider the quotient Lie algebra $L := F / I$ with the induced automorphism $\map{\overline{\alpha}}{L}{L}$. By construction, we have $r(\overline{\alpha}) = 0_L.$ \newline 
	
	In order to prove that $c(L) \geq k$, we may assume that $L$ has coefficients in the complex numbers. Indeed, the larger Lie algebra $L^{\C} := L \otimes_{\Q} \C$ over the complex numbers satisfies $c(L) = c(L^{\C})$ and it naturally admits the automorphism $\overline{\alpha}^{\C} := \overline{\alpha} \otimes \id$ satisfying $r(\overline{\alpha}^{\C}) = 0_{L^{\C}}$. \newline
	
	Let $V$ be the (complex) span of the generators $x_{1,1},\ldots,x_{2,d}$. Since $r(t)$ has no repeated roots, the operator $C \in \operatorname{GL}_{d}(\Q)$ can be diagonalised over $\C$. So we may choose an ordered eigenbasis $(y_{1,1},\ldots,y_{1,d},\ldots,y_{2,d})$ of $V$ and scalars $\lambda_1,\ldots,\lambda_d,\mu_1,\ldots,\mu_d \in \C$ such that, for each $i \in \{ 1,2\}$ and $j \in \{ 1,\ldots,d \}$, we have $\alpha(y_{1,j}) = \lambda_j \cdot y_{1,j} \text{ and } \alpha(y_{2,j}) = \mu_j \cdot y_{2,j}.$ It is clear that $\{ \lambda_1 , \ldots , \lambda_d \} = \{ \mu_1,\ldots,\mu_d \}$ is the set of roots of $r(t)$. After permuting these basis vectors, may further assume that 
	\begin{equation} r(\mu_2) = r(\lambda_1) = r(\lambda_1 \cdot \mu_2) = \cdots = r(\lambda_1 \cdot \mu_2^{k-1}) = 0 \label{EqRootsAF} . \end{equation}
	Let us define a partial order on the elements of $\operatorname{Mat}_{2,d}(\Z)$. For $a = (a_{i,j})_{i,j},b = (b_{i,j})_{i,j} \in \operatorname{Mat}_{2,d}(\Z)$ we write $a \leq b$ if and only if $a_{1,1} \leq b_{1,1} , \ldots, a_{2,d} \leq b_{2,d}$. For each element $a \in \operatorname{Mat}_{2,d}(\Z)$ satisfying $0 \leq a$, we define the family $\mathcal{B}(a)$ of all left-normed Lie monomials in the eigenvectors $y_{1,1},\ldots,y_{2,d}$ such that each $y_{i,j}$ appears with multiplicity exactly $a_{i,j}$. For the remaining $a$, we define $\mathcal{B}(a) := \emptyset$. If $F(a) = \langle \mathcal{B}(a) \rangle$ denotes the (complex) splan of $\mathcal{B}(a)$, then we naturally obtain the grading \begin{equation}
	F = \bigoplus_{a \in \operatorname{Mat}_{2,d}(\Z)} F(a)
	\label{GradingFree} \end{equation}
	of the Lie algebra $F$ by the grading group $(\operatorname{Mat}_{2,d}(\Z),+)$. \newline
	
	In order to understand the structure of the ideal $I$, we introduce some notation. For left-normed monomials $[v_1,\ldots,v_i]$ and $[w_1,\ldots,w_j]$, we define the expression $$[[v_1,\ldots,v_i];[w_1,\ldots,w_j]] := [v_1,\ldots,v_i,w_1,\ldots,w_j].$$ For each $a \in \operatorname{Mat}_{2,d}(\Z)$, we define the $\C$-span
	\begin{equation} I(a) = \sum_{\substack{0 \leq b , c \in \operatorname{Mat}_{2,d}(\Z)\\c < b + c = a}} r(\Lambda_b) \cdot \langle [ v ; w] \mid v \in \mathcal{B}(b) , w \in \mathcal{B}(c) \rangle, \label{EqHomogeneousComponent} \end{equation} 
	where $\Lambda_b := \left( \prod_{1 \leq j \leq d}  \lambda_j^{b_{1,j}} \right) \cdot \left( \prod_{1 \leq j \leq d}  \mu_j^{b_{2,j}} \right) \in \C.$
	By construction, we have $I = \sum_{a \in \operatorname{Mat}_{2,d}(\Z)} I(a)$. Since we also have the inclusions $I(a) \subseteq F(a)$, we derive from \eqref{GradingFree} the direct sum decomposition $
	I = \bigoplus_{a \in \operatorname{Mat}_{2,d}(\Z)} I(a).
	$ 	Since $I(a) \subseteq F(a)$, we conclude that the Lie algebra $L$ is also graded:
	$
	L = \bigoplus_{a \in \operatorname{Mat}_{2,d}(\Z)} L(a),
	$
	with homogeneous components $L(a) := F(a) / I(a)$. Since $$\Gamma_k(L) = \bigoplus_{\substack{0 \leq a \in \operatorname{Mat}_{2,d}(\Z)\\\sum_{i,j} a_{i,j} = k}} L(a),$$ we need only show that there exists an $a \in \operatorname{Mat}_{2,d}(\Z)$ with $\sum_{i,j} a_{i,j} = k$ and $I(a) \subsetneq F(a)$. We claim that $$
	a := \left(
	\begin{array}{ccccc}
	1 & 0 & 0 & \cdots & 0 \\
	0 & k-1 & 0 & \cdots & 0 
	\end{array}
	\right)
	\in \operatorname{Mat}_{2,d}(\Z)
	$$
	is such an element. If we define the monomials 
	\begin{eqnarray*}
		v_1 &:=& [y_{1,1},y_{2,2},\ldots,y_{2,2}] \\
		v_2 &:=& [y_{2,2},y_{1,1},y_{2,2},\ldots,y_{2,2}] \\
		&\vdots & \\
		v_{k} &:=& [y_{2,2},\ldots,y_{2,2},y_{1,1}],
	\end{eqnarray*}
	of length $k$, then $\mathcal{B}(a) = \{ v_1 , v_2 ,\ldots, v_k \}$.  
	Now \eqref{EqHomogeneousComponent} implies that
	\begin{eqnarray*} I(a) & = & \langle r(\lambda_1) \cdot v_1 , r(\lambda_1 \cdot \mu_2) \cdot v_1 , \ldots , r(\lambda_1 \cdot \mu_2^{k-1}) \cdot v_{1} , \\
		& &  r(\mu_2) \cdot v_2 , r(\lambda_1 \cdot \mu_2) \cdot v_2 , \ldots , r(\lambda_1 \cdot \lambda_2^{k-1}) \cdot v_{2} , \\
		& &
		r(\mu_2) \cdot v_3 , r(\mu_2^2) \cdot v_3 , \ldots, r(\lambda_1 \cdot \mu_2^{k-1}) \cdot v_3 , \\
		& & \ldots , \\
		& & 
		r(\mu_2) \cdot v_k , r(\mu_2^2) \cdot v_k , \ldots, r(\lambda_1 \cdot \mu_2^{k-1}) \cdot v_k \rangle . \end{eqnarray*}
	Since the anti-symmetry of the Lie bracket implies $v_3 = v_4 = \cdots = v_k = 0$, we may use \eqref{EqRootsAF} in order to conclude that $ I(a) = \{ 0\} \subsetneq \langle v_1 , v_2 , \ldots, v_k \rangle = F(a) .$ This finishes the proof.
\end{proof}

\begin{remark} \label{RemarkRefinedConstruction}
	If $\lambda \neq \mu$, then the above proof can be simplified by replacing the free $k$-step nilpotent Lie algebra on $2d$ generators with the free $k$-step nilpotent Lie algebra on $d$ generators, and by replacing the operator $C :=  A \oplus A$ with the operator $C := A$. In this case, the resulting Lie algebra will have all the correct properties, but it will have a strictly smaller dimension. The details are straightforward and we omit them. Cf. example \ref{ExGold-1}. 
\end{remark}

\begin{remark}
	This proof was inspired, in part, by Higman's construction in \cite{HigmanGroupsAndRings} of fix-point-free automorphisms of prime order on groups of prescribed class. But it is also closely related to the so-called Auslander---Scheuneman relations for the construction of semi-simple Anosov automorphisms (cf. Payne's construction in \cite{PayneAnosov}).
\end{remark}

We now consider the Mal$'$cev-correspondence:

\begin{proposition} \label{PropLiftMalcev}
	Consider a finite-dimensional, nilpotent Lie algebra $L$ over the rational numbers, together with an automorphism $\map{\gamma}{L}{L}.$ Suppose that for every lower central factor $\Gamma_i(L) / \Gamma_{i+1}(L)$, we are given a monic polynomial $r_i(t) \in \Z[t]$ such that the induced automorphism $\map{\gamma_i}{\Gamma_i(L) / \Gamma_{i+1}(L)}{\Gamma_i(L) / \Gamma_{i+1}(L)}$ satisfies $r_i(\gamma_i) = 0_{\Gamma_i(L) / \Gamma_{i+1}(L)}$. $(i.)$ Then $s(t) := r_1(t) \cdots r_{c(L)}(t)$ is a monic identity of the automorphism $\map{\operatorname{exp}(\gamma)}{\operatorname{exp}(L)}{\operatorname{exp}(L)}.$ $(ii.)$ Then the characteristic polynomial $\chi(t)$ of $\gamma$ divides a natural power of $s(t)$ and $\chi(t)$ has integer coefficients.
\end{proposition}

\begin{proof}
	$(i.)$: Let us abbreviate $G := \operatorname{exp}(L)$ and $\beta := \operatorname{exp}(\gamma)$. We recall that the Baker---Campbell---Hausdorff formula defines the group operation on $G$. This formula implies, in particular,  that the induced automorphisms $\map{\gamma_i}{\Gamma_i(L) / \Gamma_{i+1}(L)}{\Gamma_i(L) / \Gamma_{i+1}(L)}$ and $\map{\beta_i}{\Gamma_i(G) / \Gamma_{i+1}(G)}{\Gamma_i(G) / \Gamma_{i+1}(G)}$ on the lower central factors coincide. So $r_i(t)$ is an identity of $\beta_{\Gamma_i(G) / \Gamma_{i+1}(G)}$. We may now apply lemma \ref{LemmaComposingIdentities}. 	$(ii.)$: Since $r_i(\gamma_i) = 0_{\Gamma_i(L) / \Gamma_{i+1}(L)}$, the characteristic polynomial $\chi_i(t) = \det(t \cdot \id_{\Gamma_i(G) / \Gamma_{i+1}(G)} - \gamma_i) \in \Q[t]$ of $\gamma_i$ divides a natural power of $r_i(t)$. Since $r_i(t)$ is monic with integer coefficients, Gauss' lemma tells us that $\chi_i(t)$ has integer coefficients as well. Since each term $\Gamma_i(L)$ of the lower central series of $L$ is invariant under $\gamma$, the characteristic polynomial $\chi(t)$ of $\gamma$ is just the product $\chi_1(t) \cdots \chi_{c(L)}(t)$. This suffices to prove the second claim.
\end{proof}

\begin{proposition} \label{TheoremConstructionAutomorphism}	Consider a monic polynomial $r(t) \in \Z[t] \setminus \{ 1\}$. Let $\lambda$ and $\mu$ be roots of $r(t)$ in $\overline{\Q}$, and let $k$ be any natural number satisfying $r(\lambda \cdot \mu) = \cdots = r(\lambda \cdot \mu^{k-1}) = 0.$ Then $r(t)^k$ is an identity of an endomorphism $\map{\beta}{N}{N}$ of a finitely-generated, torsion-free, $k$-step nilpotent group $N$. \end{proposition}

\begin{proof} We may assume that $r(0) \neq 0$, since otherwise we may simply consider a finitely-generated, free nilpotent group $F$ of class $k$ and the endomorphism $\mapl{\gamma}{F}{F}{x}{1_F}$. 	So we may apply proposition \ref{PropConstrEndoTorsionFreeLie} in order to find a finitely-generated, $k$-step nilpotent Lie algebra $L$ over the rationals and automorphism $\map{\overline{\alpha}}{L}{L}$ satisfying $r(\overline{\alpha}) = 0_L$. Let us consider the torsion-free, $k$-step nilpotent, divisible group $G := \operatorname{exp}(L)$ corresponding with $L$, together with the automorphism $\beta := \operatorname{exp}(\overline{\alpha})$ 	of $G$ corresponding with $\overline{\alpha}$. According to proposition \ref{PropLiftMalcev}, the polynomial $r(t)^k$ is a monic identity of $\beta$. We see, in particular, that if $N$ is any $\beta$-invariant, full subgroup of $G$, then $r(t)^k$ is an identity of the restriction $\map{\beta_N}{N}{N}$. Now, since the characteristic polynomial $\chi(t)$ of $\overline{\alpha}$ has integer coefficients (cf. proposition \ref{PropLiftMalcev}), we may apply theorem $6.1$ of \cite{Dere} in order to obtain the desired subgroup $N$ of $G$. This finishes the proof. \end{proof}

\begin{proof}(Proposition \ref{PropLongProgLieAlgGroup}) $(i.)$ We simply apply proposition \ref{TheoremConstructionAutomorphism}. $(ii.)$ As before, we may assume that $r(0) \neq 0$. We first construct the finitely-generated, torsion-free, $k$-step nilpotent group $N$ and automorphism $\map{\beta}{N}{N}$ of proposition \ref{TheoremConstructionAutomorphism}. A well-known result of Gruenberg then tells us that the group $N$ is residually-(a finite $p$-group). Since $N$ is finitely-generated, there exists a characteristic subgroup $S$ of $p$-power index in $N$ such that $P := N/S$ is a finite $p$-group of class $k$. Let $\map{\gamma}{P}{P}$ be the induced automorphism. Since $r(t)^k$ is an identity of $\beta$, it is clear that $r(t)^k$ is also an identity of $\gamma$. Finally, suppose that $p$ does not divide $r(0) \cdot r(1)$. If $\gamma(x) = 1_P$ resp. $\gamma(x) = x$, then $x^{r(0)^k} = 1_P$ resp. $x^{r(1)^k} = 1_P$, and therefore $x = 1_P$. We conclude that $\gamma$ is a fix-point-free automorphism of $P$.
\end{proof}

\section{The invariants} \label{SectionInvars}

\subsection{Preliminaries}

\begin{lemma} \label{LemmaGCDWithPowers}
	Let $a(t),b(t) \in \Q[t]$ and $u \in \N$. Define $A(t) := a(t^u)$ and $B(t) := b(t^u)$. Define $h(t) := \gcd_{\Q[t]}(a(t),b(t))$ and $H(t) := \gcd_{\Q[t]}(A(t),B(t))$. Then $h(t^u) = H(t)$.
\end{lemma}

\begin{proof}
	There exists some $f(t),g(t) \in \Q[t]$ such that $a(t) = h(t) \cdot f(t)$ and $b(t) = h(t) \cdot g(t)$. By substituting $t \mapsto t^u$, we obtain $A(t) = h(t^u) \cdot f(t^u)$ and $B(t) = h(t^u) \cdot g(t^u)$. So $h(t^u) | H(t)$ in $\Q[t]$. According to Bezout's theorem, there exist $v(t),w(t) \in \Q[t]$ such that $h(t) = v(t) \cdot a(t) + w(t) \cdot b(t)$. By substitution, we obtain $h(t^u) = v(t^u) \cdot A(t) + w(t^u) \cdot B(t)$. So $H(t) | h(t^u)$ in $\Q[t]$. Since $h(t^u)$ and $H(t)$ are monic, we conclude $h(t^u) = H(t)$.
\end{proof}

\begin{lemma}[Gauss]
	Let $a(t) \in \Z[t]$ be primitive, $b(t) \in \Z[t]$. If $a(t) | b(t)$ in $\Q[t]$, then also $a(t) | b(t) $ in $\Z[t]$.
\end{lemma}

\subsection{The invariant $\tcn{r(t)}$}

\begin{definition}[$\tcn{r(t)}$]  \label{DefCongnr}
	Consider a polynomial $r(t) := a_0 + a_1 \cdot t + \cdots + a_d \cdot t^d \in \Z[t]$. For all integers $u > j \geq 0$, we define the partial sum $r_{u,j}(t) := \sum_{i \equiv j \operatorname{mod} u} a_i \cdot t^i,$ so that we obtain the periodic decomposition $r(t) = r_{u,0}(t) + \cdots + r_{u,u-1}(t)$ of $r(t)$. \label{DefPCN}
	We define the \emph{periodic congruence number} $\tcn{r(t)}$ of $r(t)$ to be the (unique) non-negative generator of the (principal) $\Z$-ideal
	\begin{equation*} \Z \cap \bigcap_{1 < u \leq \deg(r(t)) + 1} ( r_{u,0}(t) \cdot \Z[t] + \cdots + r_{u,u-1}(t) \cdot \Z[t]) \label{DefCong} .\end{equation*} 
\end{definition}

\begin{proposition} \label{PropositionCongNonZero}
	Let $r(t) \in \Z[t]$. Then the following are equivalent:
	\begin{enumerate}[(i.)]
		\item $\tcn{r(t)} = 0$.
		\item $r(0) = 0$ or there exists some $u \in \N$ and some $s(t) \in \Z[t] \setminus \Z$ such that $s(t^{u+1}) | r(t)$ in the ring $\Z[t]$.
	\end{enumerate}
\end{proposition}

\begin{proof}
	$(ii.) \implies (i.)$ Suppose first that $s(t^{u+1})|r(t)$. Then there exists a $v(t) \in \Z[t]$ such that $r(t) = s(t^{u+1}) \cdot v(t)$. Let $v(t) = \sum_j v_{u+1,j}(t)$ be the corresponding decomposition of $v(t)$. Then $r(t) = v_{u+1,0}(t) \cdot s(t^{u+1}) + v_{u+1,1}(t) \cdot s(t^{u+1}) + \cdots + v_{u+1,u}(t) \cdot s(t^{u+1})$ is the partial decomposition of $r(t)$. So $s(t^{u+1})$ divides each $r_{u+1,j}(t)$ and therefore their polynomial combination $\tcn{r(t)}$. Since $\deg(s(t^{u+1})) > 0$ and $\deg(\tcn{r(t)}) = 0$, we conclude that $\tcn{r(t)} = 0$. $(i.) \implies (ii)$ Suppose next that $\tcn{r(t)} = 0$ and $r(0) \neq 0$. Then there is a $u \in \N$ such that $h(t) := \gcd_{\Q[t]}(r_{u+1,0}(t) , \ldots , r_{u+1,u}(t))$ has degree $>0$. Since $r(0) \neq 0$, we see that $t$ does not divide $h(t)$. So $h(t) = \gcd_{\Q[t]}(r_{u+1,0}(t) \cdot t^{-0}, r_{u+1,1}(t) \cdot t^{-1}\ldots , r_{u+1,u}(t) \cdot t^{-u})$. But each of the $r_{u+1,i}(t) \cdot t^{-i}$ is in $\Z[t^{u+1}]$. Lemma \ref{LemmaGCDWithPowers} now implies that $h(t) \in \Q[t^{u+1}] \setminus \Q$, say $h(t) = S(t^{u+1})$ for some $S(t) \in \Q[t] \setminus \Q$. We can therefore write $S(t) = s(t) / m$, where $m \in \N$ and $s(t) \in \Z[t]$ is primitive. Then $s(t^{u+1}) = m \cdot h(t)$ divides each $r_{u+1,j}(t)$, and therefore their sum $r_{u+1,0}(t) + \cdots + r_{u+1,u}(t) = r(t)$ in $\Q[t]$. Gauss' lemma now implies that also $s(t^{u+1}) | r(t)$ in $\Z[t]$.
\end{proof}

\begin{corollary}
	Let $r(t)$ be a good polynomial. Then $\tcn{r(t)} \neq 0$.
\end{corollary}

\subsection{The invariants $\discr{r(t)}$ and $\prodant{r(t)}$} \label{SubSecInvarsDiscrProd}

\begin{definition}[$\discr{r(t)}$ and $\prodant{r(t)}$] \label{DefDiscProd}
	Consider a polynomial $r(t) \in \Z[t]$. If $r(t)$ is constant, we define $\discr{r(t)} := r(t)$ and $\prodant{r(t)} := 1$. Else, we let $a$ be the leading coefficient of $r(t)$, we let $\lambda_1,\ldots,\lambda_l$ be the distinct roots of $r(t)$ with corresponding multiplicities $m_1,\ldots,m_l$, and we set $m := \max \{ m_1,\ldots,m_l\}$. We then define 
	\begin{eqnarray*}
		\discr{r(t)} := a^{1+2d^2} \cdot (m-1)! \cdot \prod_{\substack{1 \leq i , j \leq l\\i \neq j}} (\lambda_i - \lambda_j)^m \label{DefDisc}
	\end{eqnarray*}
	and
	\begin{eqnarray*}
		\prodant{r(t)} := a^{2d^3} \cdot \prod_{\substack{1 \leq i , j \leq l\\r(\lambda_i \cdot \lambda_j) \neq 0}} r(\lambda_i \cdot \lambda_j) = a^{2d^3} \cdot \prod_{\substack{1 \leq i , j , k \leq l\\r(\lambda_i \cdot \lambda_j) \neq 0}} a \cdot (\lambda_i \cdot \lambda_j - \lambda_k)^{m_k}
		. \label{DefProd}
	\end{eqnarray*}
\end{definition}

Let us now verify that these invariants are non-zero integers (for $r(t) \in \Z[t] \setminus \{0\}$) and let us show how they can be computed algorithmically.

\paragraph{$\discr{r(t)}$.} When considering $\discr{r(t)}$, we may assume that $r(t)$ has \emph{degree at least $2$}, since otherwise the computation is quite straight-forward. We may then use the standard algorithms to compute the square-free factorisation $r(t) = u_1(t)^1 \cdot u_2(t)^2 \cdots u_n(t)^n$ of $r(t)$ in $\Z[t]$, with the convention that $n$ be minimal. Then $u(t) := u_1(t) \cdots u_n(t)$ is a greatest square-free factor $u(t) := u_1(t) \cdots u_n(t)$ of $r(t)$ in $\Z[t]$, and it is unique up to its sign. Then $m = n$, $l = \deg(u(t))$, and the leading coefficient $\bar{a}$ of $u(t)$ divides $a$ in $\Z$. So the polynomial $v(t) := (a/\bar{a}) \cdot u(t) = a \cdot (t - \lambda_1) \cdots (t - \lambda_l)$ has integer coefficients. Let $\operatorname{Syl}(v(t),v'(t))$ be the Sylvester matrix of $v(t)$ and its formal derivative $v'(t)$. 

\begin{lemma} \label{LemmaIntegrDiscrast} We have \begin{eqnarray}
	\operatorname{Discr}_\ast(r(t)) &=& a^{1+2d^2 - 2 m(l-1)-m} \cdot (m-1)! \cdot \left( \det \left( \operatorname{Syl}(v(t),v'(t)) \right) \right)^m
	\label{FormulaExplicitDiscrSylvester}
	\end{eqnarray}
	and $\discr{r(t)} \in \Z \setminus\{ 0 \} .$ 
\end{lemma}

\begin{proof}
	By using the formula $\discr{v(t)} = (-1)^{l(l-1)} \cdot a^{2l-2} \cdot \prod_{1 \leq i \neq j \leq l} (\lambda_i - \lambda_j)$, we obtain 
	$
	\discr{r(t)} / ( (m-1)! \cdot (\operatorname{Discr}(v(t)))^m ) = (-1)^{m l (l-1)/2} \cdot a^{1+d^2-2m(l-1)}.
	$
	Since $m,l \leq d$, we have $0 \leq 1+2d^2 - 2m(l-1)$ and therefore $a^{1+2d^2-2m(l-1)} \in \Z \setminus \{ 0 \}$. Since $v(t) \in \Z[t]$, we also have $\operatorname{Discr}(v(t)) \in \Z \setminus\{ 0 \}$. So we may indeed conclude that $\discr{r(t)} \in \Z \setminus \{ 0 \}$. Finally, by using the formula $
	\operatorname{Discr}(v(t)) = (-1)^{l(l-1)/2} \cdot a^{-1} \cdot \det \left( \operatorname{Syl}(v(t),v'(t)) \right) $	we obtain \eqref{FormulaExplicitDiscrSylvester}. 
\end{proof}

\begin{remark}
	Formula \eqref{FormulaExplicitDiscrSylvester} allows us to compute $\discr{r(t)}$ \emph{without having to extract roots}. Indeed: the invariants $m$, $v(t)$, and $l$ are given by the algorithm for the square-free factorisation of $r(t)$.
\end{remark}

\paragraph{$\prodant{r(t)}$.} Let us now find a similar formula for $\prodant{r(t)}$. Let $C$ be the companion matrix of $(1/a) \cdot v(t)$ and let us consider its Kronecker square $C \otimes C$. The corresponding characteristic polynomial is given by $\chi_{C \otimes C}(t) = \prod_{1 \leq i , j \leq l}(t - \lambda_i \cdot \lambda_j).$ We may next use the Euclidean algorithm in order to compute a greatest factor $w(t)$ of $\chi_{C \otimes C}(t)$ (in the ring $\Q[t]$) that is co-prime to $r(t)$. Since such a factor $w(t)$ is determined only up to a (non-zero) rational number, we may choose the unique $w(t)$ with leading coefficient $a^{2\deg(w(t))}$. We then have the factorisation $$w(t) = \prod_{\substack{1 \leq i , j \leq l\\r(\lambda_i \cdot \lambda_j) \neq 0}} a^2 \cdot (t - \lambda_i \cdot \lambda_j).$$

\begin{lemma} \label{LemmaResultantProdant} We have 
	\begin{eqnarray}
	\prodant{r(t)} &=& a^{2(d^2-\deg(w(t)))d} \cdot \det \left( \operatorname{Syl}(r(t),w(t)) \right)
	\label{FormulaForProdExplicit}
	\end{eqnarray}
	and $\prodant{r(t)} \in \Z \setminus \{ 0 \}.$
\end{lemma}

\begin{proof}
	We define the auxiliary, monic polynomials
	\begin{equation} \bar{r}(t) := \prod_{\substack{ 1 \leq i , j \leq l \\ r(\lambda_i \cdot \lambda_j) \neq 0 }} (t - a^2 \cdot \lambda_i \cdot \lambda_j)  \text{ and } \bar{\bar{r}}(t) := \prod_{\substack{ 1 \leq i , j \leq l \\ r(\lambda_i \cdot \lambda_j) = 0 }} (t - a^2 \cdot  \lambda_i \cdot \lambda_j) \nonumber.\end{equation} These polynomials have rational coefficients since their roots are permuted by every automorphism $\sigma \in \operatorname{Gal}(\Q(\lambda_1,\ldots,\lambda_l) : \Q)$. In order to prove that $\bar{r}(t) \in \Z[t]$, we next consider the auxiliary polynomial $s(t;t_1,\ldots,t_l) := \prod_{1 \leq i , j \leq l} (t - t_i \cdot t_j) $ in the variable $t$ with coefficients in the domain  $\Z[t_1,\ldots,t_l]$. According to Viet\`a's formula, we have $$s(t;t_1,\ldots,t_l) = \sum_{0 \leq k \leq l^2} (-1)^k \cdot e_{k}^{[l^2]} ( t_1 \cdot t_1 , t_1 \cdot t_2 , \cdots , t_l \cdot t_l) \cdot t^{l^2-k} ,$$ where each $e_{j}^{[i]}(t_1,\ldots,t_i) := \sum_{1 \leq n_1 < \cdots < n_j \leq i} t_{n_1} t_{n_2} \cdots t_{n_j}$ is the elementary symmetric polynomial in the variables $t_1,\ldots,t_i$. We note that each coefficient $e_k^{[l^2]}(t_1^2,\ldots,t_l^2)$ in this expression is a also symmetric polynomial \emph{in the variables $t_1,\ldots,t_l$} with coefficients in $\Z$. So, according to the fundamental theorem for symmetric functions, there exist polynomials $P_1(t_1,\ldots,t_l) , \ldots , P_{l^2}(t_1,\ldots,t_l) \in \Z[t_1,\ldots,t_{l}]$ such that $$s(t;t_1,\ldots,t_l) = \sum_{0 \leq k \leq l^2} P_k(e_{1}^{[l]}(t_1,\ldots,t_l),\ldots,e_{l}^{[l]}(t_1,\ldots,t_l))  \cdot t^{l^2-k} .$$ By evaluating $t_i \mapsto a \cdot \lambda_i$, we obtain
	\begin{eqnarray*}
		s(t;a \cdot \lambda_1, \ldots , a \cdot \lambda_l) &=& \sum_{0 \leq k \leq l^2} P_k(a \cdot e_{1}^{[l]}(\lambda_1,\ldots,\lambda_l),\ldots,a^l \cdot e_{l}^{[l]}(\lambda_1,\ldots,\lambda_l))  \cdot t^{l^2-k} . \label{AuxFormProd}
	\end{eqnarray*}
	Since $u(t) = \sum (-1)^k \cdot {a} \cdot e_{k}^{[l]}(\lambda_1,\ldots,\lambda_l) \cdot t^{l-k}$ and $P_1(t_1,\ldots,t_l), \ldots, P_{l^2}(t_1,\ldots,t_l) $ all have integer coefficients, we see that the monic polynomial $s(t;a \cdot \lambda_1, \ldots , a \cdot \lambda_l)$ also has integer coefficients. Since $s(t;a \cdot \lambda_1,\ldots,a \cdot \lambda_l)$ is a product $s(t;a \cdot \lambda_1,\ldots,a \cdot \lambda_l) = \bar{r}(t) \cdot \bar{\bar{r}}(t)$ of two monic polynomials with rational coefficients, we may use Gauss' lemma to conclude that also $\bar{r}(t) , \bar{\bar{r}}(t) \in \Z[t]$. We see, in particular, that $w(t) = \bar{r}(a^2 \cdot t) \in \Z[t]$, so that $\operatorname{Res}(r(t),w(t)) \in \Z$. We now observe that
	\begin{eqnarray*}
		\prodant{r(t)} &:=& a^{2d^3} \cdot \prod_{\substack{1 \leq i , j \leq l \\ r(\lambda_i \cdot \lambda_j) \neq 0}} r(\lambda_i \cdot \lambda_j) \nonumber \\
		&=& a^{2(d^2-\deg(w(t)))d} \cdot \operatorname{Res}(r(t),w(t)) \label{FormulaForProd}
	\end{eqnarray*}
	is a non-zero integer. By using the formula 
	$ \operatorname{Res}(r(t),w(t)) 
	= \det \left( \operatorname{Syl}(r(t),w(t)) \right) $, 
	we finally obtain formula \eqref{FormulaForProd}.
\end{proof}

\begin{remark}
	Formula \eqref{FormulaForProd} allows us to compute $\prodant{r(t)}$ \emph{without having to extract roots}. Indeed: we only need the algorithm for the square-free factorisation of $r(t)$ in $\Z[t]$ and the Euclidean algorithm in the ring $\Q[t]$.
\end{remark}

\subsection{Arithmetically-free sets of roots}

\begin{definition}[Arithmetically-free subsets \cite{MoensAF}]
	Let $(A,\cdot)$ be an abelian group and let $X$ be a finite subset of $A$. We say that $X$ is arithmetically-free subset of $A$ if and only if $X$ contains \emph{no} arithmetic progressions of the form $\lambda , \lambda \cdot \mu , \lambda \cdot \mu^2 , \ldots , \lambda \cdot \mu^{|X|}$ with $\lambda,\mu \in X$.
\end{definition}

\begin{proposition}[Arithmetically-free sets of roots] \label{PropositionArithmeticallyFreeSetsOfRoots}
	Consider a polynomial $r(t) \in \Z[t]$ satisfying $r(0) \neq 0$. Let $X$ be the set of roots in $\overline{\Q}^\times$. Then the following are equivalent.
	\begin{enumerate}[(i.)]
		\item $X$ is \emph{not} an arithmetically-free subset of $(\overline{\Q}^\times,\cdot)$.
		\item There is some $u \in \N$ and some $s(t) \in \Z[t] \setminus \Z$ such that $\Phi_u(t)$ and $s(t^{u})$ divide $r(t)$ in $\Z[t]$.
	\end{enumerate}
\end{proposition}

\begin{proof}
	$(ii.) \implies (ii.)$ Suppose that the second claim holds. Let $\lambda$ be any root of $s(t^u)$ and let $\mu$ be any primitive $u$'th root of unity. Then $r(\mu) = 0$ and $r(\lambda) = r(\lambda \cdot \mu) = r(\lambda \cdot \mu^2) = \cdots$ is an arithmetic progression of roots in $X$ with $\mu \in X$. $(i.) \implies (ii.)$ Now suppose that $X$ is not an arithmetically-free subset of $\overline{\Q}^\times$. Then there exist $\lambda,\mu \in X$ such that $\lambda , \ldots , \lambda \cdot \mu^{|X|} \in X$. So $\mu$ has finite order, say $u \geq 1$, and $\Phi_u(t) | r(t)$ in $\Z[t]$. We now consider the partial decomposition $r(t) = r_{u,0}(t) + \cdots + r_{u,u-1}(t)$ of $r(t)$ and evaluate in the terms of the progression. We obtain the linear Vandermonde system
	\[
	\left(
	\begin{array}{cccc}
	1  & 1  &  \cdots & 1 \\
	1  & \mu  & \cdots & \mu^{u-1} \\
	\vdots  & \vdots  & \ddots & \vdots \\
	1  & \mu^{u-1} & \cdots & (\mu^{u-1})^{u-1}
	\end{array}
	\right)
	\cdot 
	\left(
	\begin{array}{c}
	r_{u,0}(\lambda)\\
	r_{u,1}(\lambda)\\
	\vdots\\
	r_{u,u-1}(\lambda)
	\end{array}
	\right)
	=
	\left(
	\begin{array}{c}
	0\\
	0\\
	\vdots\\
	0
	\end{array}
	\right).
	\]
	Since the order of $\mu$ is $u$, the determinant of this system is non-zero. So the partial sums $r_{u,j}(\lambda)$ all vanish. Since $\lambda \neq 0$, we see that $\lambda$ is a common root of $r_{u,0}(t) \cdot t^{-0} , r_{u,1}(t) \cdot t^{-1} , \ldots , r_{u,u-1}(t) \cdot t^{-(u-1)}$ and their greatest common divisor, say $h(t)$. According to lemma \ref{LemmaGCDWithPowers}, there is an $S(t) \in \Q[t] \setminus \Q$ such that $h(t) = S(t^u)$. Write $S(t)$ as $s(t) / m $, for some $m \in \N$ and some primitive $s(t) \in \Z[t]$. Gauss' lemma then implies that $s(t^u)$ divides $r(t)$ in $\Z[t]$.
\end{proof}

\begin{corollary} \label{CorollaryGoodImpliesAF}
	Let $r(t) \in \Z[t]$ be a good polynomial. Then its roots form an arithmetically-free subset $X$ of $(\overline{\Q}^\times,\cdot)$.
\end{corollary}

\begin{proof}
	Suppose that $X$ is \emph{not} arithmetically-free. Since $r(t)$ is good, we have $r(0) \neq 0$. So proposition \ref{PropositionArithmeticallyFreeSetsOfRoots} gives a $u \in \N$ and $s(t) \in \Z[t] \setminus \Z$ such that $\Phi_u(t) , s(t^u) | r(t)$. Since $r(t)$ is good we know that $u < 2$. Since $r(t)$ is good, we have $r(1) \neq 0$ and therefore $\Phi_u(1) \neq 0$. So $u \geq 2$. This contradiction finishes the proof.
\end{proof}

\section{Examples} \label{SectionExamplesForLater}

\subsection{$r(t) = t^3-2t-1$} 

We now illustrate the methods of the previous sections by means of single, concrete example. We begin by illustrating the construction of \ref{SubSecConstrIdentity}.

\begin{example} \label{HeisenbergIdentity} We consider the discrete Heisenberg group $H$, defined as the subgroup:
	$$ H := \left\{ \left(
	\begin{array}{ccc}
	1 & x & z \\ 
	0 & 1 & y \\ 
	0 & 0 & 1 \\ 
	\end{array}
	\right)
	|
	x,y,z \in \Z
	\right\} \subseteq \operatorname{GL}_3(\Z).
	$$ Let $\map{\gamma}{H}{H}$ be the automorphism that is given by
	$$ \left(
	\begin{array}{ccc}
	1 & x & z \\ 
	0 & 1 & y \\ 
	0 & 0 & 1 \\ 
	\end{array}
	\right)
	\mapsto
	\left(
	\begin{array}{ccc}
	1 & y & x \cdot y + \frac{y \cdot (y-1)}{2} - z \\ 
	0 & 1 & x + y \\ 
	0 & 0 & 1 \\ 
	\end{array}
	\right),
	$$
	and let us use the series $H \geq [H,H] \geq \{ \id_3 \}$ with factors $H / [H,H] \cong \Z^2$ and $[H,H] / \{ \id_3 \} \cong \Z$. A straight-forward computation gives us the characteristic polynomials $\chi_1(t) = -1 -t+t^2$ and $\chi_2(t) = 1+t$, so that the characteristic polynomial of $\gamma$ with respect to the series is given by $\chi(t) := (-1-t+t^2) \cdot (1 + t) $. By substitution, as in the proof of Proposition \ref{TheoremCH}, we obtain $\forall v \in H : \phantom{} \gamma^{3}(v) \cdot \gamma^{2}(v^{-1}) \cdot \gamma(v^{-1}) \cdot \gamma^{2}(v) \cdot \gamma(v^{-1}) \cdot v^{-1} = \id_3, $ so that the inverse automorphism $\map{\gamma^{-1}}{H}{H}$ is given by the formula $ \gamma^{-1}(v) = \gamma^{2}(v) \cdot \gamma^{1}(v^{-1}) \cdot (v^{-1}) \cdot \gamma(v) \cdot v^{-1} ,$ and therefore by
	$$ \left(
	\begin{array}{ccc}
	1 & x & z \\ 
	0 & 1 & y \\ 
	0 & 0 & 1 \\ 
	\end{array}
	\right)
	\mapsto
	\left(
	\begin{array}{ccc}
	1 & -x + y & x \cdot y - \frac{x \cdot (1+x)}{2} - z \\ 
	0 & 1 & x \\ 
	0 & 0 & 1 \\ 
	\end{array}
	\right).
	$$
\end{example}

We emphasize that the map $\mapl{-1 - 2 \gamma + \gamma^3}{H}{H}{v}{v^{-1} \cdot \gamma(v^{-1}) \cdot \gamma^3(v)}$ does not vanish. So $r(t)$ is an identity of $\gamma$, but it is not a monotone identity of $\gamma$. We now use (a minor variation on) the construction of \ref{SubSectionConstructionAuto} to go in the other direction: 

\begin{example} \label{ExGold-1} Let $r(t) := (t^2-t-1) \cdot (t+1)$ be the polynomial of example \ref{HeisenbergIdentity}. Let us construct a fix-point-free automorphism $\map{\beta}{N}{N}$ on a finitely-generated, torsion-free, two-step nilpotent group $N$ such that $r(t)$ is the characteristic polynomial of $\beta$. \newline
	
	The roots of $r(t)$ are $ -1 , \frac{1 - \sqrt{5}}{2} $, and $\frac{1 + \sqrt{5}}{2} $, and the product of the latter two roots is the first root (cf. remark \ref{RemarkRefinedConstruction}). So we consider the free two-step nilpotent Lie algebra $F = \Q \cdot x_1 + \Q \cdot x_2 + \Q \cdot [x_1,x_2]$ on the generators $x_1$ and $x_2$. Then the companion matrix $$
	A := 
	\left(
	\begin{array}{cc}
	0&1 \\
	1&1 
	\end{array}
	\right) \in \operatorname{GL}_4(\Z)
	$$ of $(t^2-t-1)$ defines a linear transformation of the $\Q$-span of the generators: $\alpha(x_1) := x_2$ and $\alpha(x_2) := x_1 + x_2$. This map extends (in a unique way) to an automorphism $\map{\alpha}{F}{F}$ of the Lie algebra $F$: $\alpha([x_1,x_2]) = [\alpha(x_1),\alpha(x_2)] = - [x_1,x_2]$. 	
	We let $I$ be the ideal of $F$ that is generated by the subset $(\alpha^2 - \alpha - \id_F)(\Q \cdot x_1 + \Q \cdot x_2) + (\alpha + \id_F)(\Q \cdot [x_1,x_2])$ of $F$.  Then the induced automorphism $\map{\overline{\alpha}}{L}{L}$ on the quotient Lie algebra $L := F / I$ satisfies $r(\overline{\alpha}) = 0_L$. In fact: $I = \{ 0_F\}$ and, with respect to the ordered basis $(\overline{x}_1,\overline{x}_2,\frac{\overline{[x_1,x_2]}}{2})$, the automorphism of $L$ is given by the matrix 
	$$
	\left(
	\begin{array}{cc|c}
	0&1&0 \\
	1&1&0 \\ \hline
	0&0&-1 
	\end{array}
	\right) .
	$$ 
	We may now use the Baker---Campbell---Hausdorff formula to define the group operation $\ast$ on $L$. For rational numbers $c_1,c_2,c_{12}$ and $C_{1},C_2,C_{12}$, we define 	$(c_1 \cdot \overline{x}_1 + c_2 \cdot \overline{x}_2 + c_{12} \cdot \frac{\overline{[x_1,x_2]}}{2}) \ast (C_1 \cdot \overline{x}_1 + C_2 \cdot \overline{x}_2 + C_{12} \cdot \frac{\overline{[x_1,x_2]}}{2}) $ to be $ (c_1 + C_1) \cdot \overline{x}_1 + (c_2 + C_2) \cdot \overline{x}_2 + (c_{12} + C_{12} + (c_1 \cdot C_2 - c_2 \cdot C_1) ) \cdot \frac{\overline{[x_1,x_2]}}{2} .$ One can then verify that $ N := \Z \cdot \overline{x}_1 + \Z \cdot \overline{x}_2 + \Z \cdot \frac{\overline{[x_1,x_2]}}{2} $ is an $\overline{\alpha}$-invariant subgroup of $(L,\ast)$ of class two and Hirsch-length $3$. The restriction $\map{\beta}{N}{N}$ of $\overline{\alpha}$ to $N$ is an endomorphism of $N$ and $r(t)$ is the characteristic polynomial of $\beta$ (with respect to the series of the isolators of the lower central series of $N$). Since also $r(0) = -1$, we may use remark \ref{RemarkChar0} in order to conclude that $\beta$ is, in fact, an automorphism of $N$. Since $r(1) = -2$ and since $N$ is torsion-free, we know that all fix-points of $\beta$ are trivial, so that $\beta$ is fix-point-free. In fact, the group $N$ is a twisted Heisenberg group:
	$$ N \cong \left\{ \left(
	\begin{array}{ccc}
	1 & x & \frac{z}{2} \\ 
	0 & 1 & y \\ 
	0 & 0 & 1 \\ 
	\end{array}
	\right)
	|
	x,y,z \in \Z
	\right\} \subseteq \operatorname{GL}_3(\Q).
	$$ And, under this identification, the automorphism $\map{\beta}{N}{N}$ is given by:
	$$ \left(
	\begin{array}{ccc}
	1 & x & \frac{z}{2} \\ 
	0 & 1 & y \\ 
	0 & 0 & 1 \\ 
	\end{array}
	\right)
	\mapsto
	\left(
	\begin{array}{ccc}
	1 & y & \frac{2 \cdot x \cdot y + y^2 - z}{2} \\ 
	0 & 1 & x + y \\ 
	0 & 0 & 1 \\ 
	\end{array}
	\right).
	$$
	This group $N$ is not the discrete Heisenberg group $H$ of example \ref{HeisenbergIdentity}.  But $H$ is a normal subgroup of index $2$ in $N$.
\end{example}

\begin{example}
	The polynomial $r(t) := t^3-2t-1$ is good and its roots form an arithmetically-free subset of $(\overline{\Q}^\times,\cdot)$.
\end{example}

\begin{proof}
	Clearly, $r(0) \cdot r(1) = 2$. We next observe that $r(t) $ factorizes as $ (t+1) \cdot (t^2-t-1)$. So $(t)$ is good and we may apply proposition \ref{PropositionArithmeticallyFreeSetsOfRoots}.
\end{proof}

We conclude, in particular, that the invariants $r(1), \tcn{r(t)}, \discr{r(t)} ,$ and $\prodant{r(t)}$ are non-zero. Let us now compute these invariants. 

\begin{example} \label{ExampleGoldenDP}
	For $r(t) := t^3-2t-1$, we have $r(1) \cdot \tcn{r(t)} \cdot \discr{r(t)} \cdot \prodant{r(t)} = (-2) \cdot (2) \cdot (-5) \cdot (-2^7 \cdot 5).$
\end{example}

\begin{proof}
	$(i.)$ It is clear that $r(1) = -2$. $(ii.)$ In order to compute $\tcn{r(t)}$, we use the Euclidean algorithm. We see that $ 1 = - r_{2,0}(t) + 0 \cdot r_{u,1}(t)$ and $2 = (-t^2) \cdot r_{3,0}(t) + (-2) \cdot r_{3,1}(t) + 0 \cdot  r_{3,2}(t) $. So $\tcn{r(t)} | 2$. Suppose that $\tcn{r(t)} = 1$. Then there exist $a(t),b(t),c(t) \in \Z[t]$ with $a(t) \cdot r_{3,0}(t) + b(t) \cdot r_{3,1}(t) + c(t) \cdot r_{3,2}(t) = 1$. By evaluating in $1$ and reducing modulo $2$, we obtain the contradiction $0 = 1$. $(iii.)$ In order to compute the invariants $\discr{r(t)}$ and $\prodant{r(t)}$, we assume the notation of \ref{SubSecInvarsDiscrProd}. The roots of this polynomial are simple and given by $\lambda_1 := -1$, $\lambda_2 := \frac{1-\sqrt{5}}{2}$ and $\lambda_3 := \frac{1+\sqrt{5}}{2}$. By definition, we therefore have $\discr{r(t)} := - (\lambda_1-\lambda_2)^2 \cdot (\lambda_2-\lambda_3)^2 \cdot (\lambda_3 - \lambda_1)^2 = -5 .$ We next note that $r(\lambda_i \cdot \lambda_j) = 0$ if and only if $\{i,j\} = \{2,3\}$. So definition \ref{DefDiscProd} gives us $$\prodant{r(t)} := \prod_{\substack{1 \leq i , j \leq 3 \\ \{i,j\} \neq \{2,3\}}} r(\lambda_i \cdot \lambda_j) = - 2^7 \cdot 5.  $$ Let us now come to the same conclusions without using the polynomial's roots. Since $r(t)$ is monic, we do not have to keep track of leading coefficients. The square-free factorisation of the monic polynomial $r(t)$ is given by $r(t) = u_1(t)^1$, so that $r(t) = u(t) = v(t)$ and $m = 1$. By using formula  \eqref{FormulaExplicitDiscrSylvester}, we obtain $\discr{r(t)} = -5$. Now let $C$ be the companion matrix of $r(t)$.	We then have $\chi_{C \otimes C}(t) = -(-1 + t) (1 + t)^2 (1 - 3 t + t^2) (-1 + t + t^2)^2$ and $w(t) = (t^2+t-1)^2 (t^2-3t+1)(t-1).$ Formula \eqref{FormulaForProdExplicit} allows us to conclude once more that $\prodant{r(t)} = -2^7 \cdot 5$.
\end{proof}

\subsection{$\Phi_n(t)$ and $\Psi_n(t)$} \label{SubsecExsPhiPsi}

We now compute the invariants of the cyclotomic polynomials $\Phi_n(t)$ and the polynomials $\Psi_n(t) = (t^n-1)/(t-1)$.

\begin{definition}
	For $r(t) \in \Z[t]$ and $u \in \N$, we define $\operatorname{RRes}_u(r(t))$ to be the (unique) non-negative generator of the (principal) $\Z$-ideal $\Z \cap ( r_{u,0}(t) \cdot \Z[t] + \cdots + r_{u,u-1}(t) \cdot \Z[t] )$. Then clearly $\tcn{r(t)}$ is the least common multiple of $\operatorname{RRes}_2(r(t)), \ldots ,\operatorname{RRes}_{u+1}(r(t))$. 
\end{definition}

We will need two elementary properties of reduced resultants.

\begin{lemma}[Division] \label{LemmaRResDivision} Consider integer polynomials $a(t),b(t),c(t) \in \Z[t]$ with $a(t) \cdot b(t) = c(t)$. For every integer $u \geq 2$, we have $ \operatorname{RRes}_u(a(t)) | \operatorname{RRes}_u(c(t)) .$ \end{lemma}

\begin{proof}  We first note that for every natural $i \geq 0$, we have \begin{equation} c_{u,i}(t) = \sum_{j + k \equiv i \operatorname{mod} u} a_{u,j}(t) \cdot b_{u,k}(t) \label{FormulaConvolution} .\end{equation} By definition, there exist polynomials $C_{0}(t),\ldots,C_{u-1}(t) \in \Z[t]$ that give us the equality $\operatorname{RRes}_u(c(t)) = \sum_{0 \leq i < u} C_{i}(t) \cdot c_{u,i}(t).$ Using \eqref{FormulaConvolution}, we obtain:
	$ \operatorname{RRes}_u(c(t)) = \sum_{0 \leq i < u} C_{i}(t) \cdot \left( \sum_{ j + k \equiv i \operatorname{mod} u} a_{u,j}(t) \cdot b_{u,k}(t) \right) . 
	$
	So $\operatorname{RRes}_u(c(t)) \in (a_{u,0}(t) \cdot \Z[t] + \cdots + a_{u,u-1}(t) \cdot {\Z[t]}) \cap \Z$. 
\end{proof}

\begin{lemma}[Composition] \label{LemmaRResComposition} Consider $r(t) \in \Z[t]$. Consider natural numbers $u \geq 2$, and $n,m$. If $m \cdot n \equiv 1 \operatorname{mod} u$, then $ \operatorname{RRes}_u(r(t^m)) | \operatorname{RRes}_u(r(t)) .$\end{lemma}

\begin{proof} Let the polynomial be given by $r(t) = \sum_{0 \leq j \leq d} a_j \cdot t^j$ and define $s(t) := r(t^m)$. Then for all $0 \leq i < u$, we have $$ s_{u,i}(t) = \sum_{\substack{0 \leq j \leq d\\m \cdot j \equiv i \operatorname{mod} u}} a_j \cdot (t^m)^j = \sum_{\substack{0 \leq j \leq d\\j \equiv n \cdot i \operatorname{mod} u}} a_j \cdot (t^m)^j = r_{u,n \cdot i \operatorname{mod} u}(t^m).$$ By definition, there exist $c_{u,0}(t) , \ldots, c_{u,u-1}(t) \in \Z[t]$ such that $\operatorname{RRes}_u(r(t)) = \sum_j c_{u,j}(t) \cdot r_{u,j}(t)$. By substituting $t \mapsto t^m$, we obtain $ \operatorname{RRes}_u(r(t)) = \sum_j c_{u,j}(t^m) \cdot r_{u,j}(t^m) \in s_{u,0}(t) \cdot \Z[t] + \cdots + s_{u,u-1}(t) \cdot \Z[t]. $ We conclude that $\operatorname{RRes}_u(r(t)) \in (s_{u,0}(t) \cdot \Z[t] + \cdots + s_{u,u-1}(t) \cdot \Z[t]) \cap \Z$. \end{proof}

Before computing $\tcn{\Phi_n(t)}$, we record some elementary properties of the cyclotomic polynomials. 

\begin{lemma}	\label{LemFormulaCyclo} Let $n > 1$ be a natural number and let $m$ be its radical. Then $\deg(\Phi_n(t)) = \phi(n)$, where $\phi$ is the Euler totient-function. Then \begin{equation} \Phi_n(t) = \Phi_m(t^{n/m}) \label{FormulaRadical}. \end{equation} Let $p$ be an odd prime that does not divide $m$. Then \begin{equation} \Phi_{p \cdot m}(t) \cdot \Phi_m(t) = \Phi_m(t^{p}) \label{FormulaDivision} . \end{equation} If $m$ is odd, then \begin{equation} \Phi_{2 \cdot m}(t) = \Phi_m(-t) . \label{FormulaEven}\end{equation} If $n = m$, then $\Phi_n(1) = m$. If $n \neq m$, then $\Phi_n(1) = 1$. \end{lemma}

\begin{lemma} \label{PropositionRResCycloSQFree} For every square-free natural number $n$ and natural number $u \geq 2$, we have $\operatorname{RRes}_u(\Phi_n(t)) = 1$. \end{lemma}

\begin{proof} Define $r(t):= \Phi_n(t)$. Since $\operatorname{RRes}_2(-1+t) = 1$, we may assume that $n > 1$. \newline
	
	\emph{Case: $n$ is a prime}. Suppose first that $u \geq \phi(n)+ 1 = (n-1) + 1 = n$. Then $r_{u,0}(t) = 1$, so that $1 = 1 \cdot r_{u,0}(t) + 0 \cdot r_{u,1}(t) + \cdots + 0 \cdot r_{u,u-1}(t)$, and therefore $\operatorname{RRes}_u(r(t)) = 1$. Next, we suppose that $2 \leq u < n$. Then $n-1 \not \equiv u-1 \mod u$, so that $1 = r_{u,0}(t) + 0 \cdot r_{u,1}(t) + \cdots + 0 \cdot r_{u,u-2}(t) - t \cdot r_{u,u-1}(t).$ This also implies that $\operatorname{RRes}_u(r(t)) = 1$. \newline
	
	\emph{Case: $n$ is square-free and odd}. Let us proceed by induction on the number $l$ of distinct prime factors of $n$. The base of the induction, $l = 1$, is given by the previous paragraph. So we suppose that $l > 1$. Let $n = p_1 \cdots p_l$ be the decomposition of $n$ into (distinct, odd) primes. We may, as before,  suppose that $u$ is an integer satisfying $2 \leq u \leq \phi(n) + 1 = (p_1-1) \cdots (p_l-1) +1$. Then we note that there is at least one $i \in \{ 1,\ldots,l\}$ such that $p_i$ does \emph{not} divide $u$. Formula \eqref{FormulaDivision} tells us that $\Phi_{n}(t) | \Phi_{n/{p_i}}(t^{p_i})$. Lemma \ref{LemmaRResDivision}, lemma \ref{LemmaRResComposition}, and the induction hypothesis then imply that 	$ \operatorname{RRes}_u(\Phi_{n}(t))  |  \operatorname{RRes}_u(\Phi_{n/{p_i}}(t^{p_i}))  
	|  \operatorname{RRes}_u(\Phi_{n/{p_i}}(t)) = 1.	$ \\
	
	\emph{Case: $n$ is square-free and even}. Formula \eqref{FormulaEven} gives us the equality $\Phi_{n}(t) = \Phi_{n/2}(-t)$. The odd case then tells us that $\operatorname{RRes}_u(\Phi_n(t)) = \operatorname{RRes}_u(\Phi_{n/2}(-t)) = \operatorname{RRes}_u(\Phi_{n/2}(t)) = 1.$ This finishes the proof. \end{proof}

\begin{proposition}[$\tcn{\Phi_n(t)}$ and $\tcn{\Psi_n(t)}$] \label{PropCongCycloComputed} Let $n \in\N$. Then we have:	
	\begin{equation*}
	\tcn{\Phi_n(t)} =
	\begin{cases}
	1 & \text{if } n \text{ is square-free},\\
	0 & \text{if } n \text{ is not square-free}.
	\end{cases}
	\end{equation*}
	Moreover: $\tcn{\Psi_n(t)} = 0$ if and only if $n$ is composite.
\end{proposition}

\begin{proof} $(i)$ Let $m$ be the radical of $n$. If $n$ is not square-free, then $\Phi_n(t) \in \Z[t^{n/m}]$, so that $\operatorname{RRes}_{n/m}(\Phi_n(t)) = 0$, and therefore $\tcn{\Phi_n(t)} = \operatorname{lcm}_{u>1} \operatorname{RRes}_{u}(\Phi_n(t)) = 0$. Else, $n$ is square-free, and we may apply proposition \ref{PropositionRResCycloSQFree} to conclude that $\tcn{\Phi_n(t)} = \operatorname{lcm}_{u>1} \operatorname{RRes}_{u}(\Phi_n(t)) = 1$. $(ii.)$ If $n$ has a proper, non-trivial divisor $u$, then we have  $\Psi_{n/u}(t) + \Psi_{n/u}(t) \cdot t + \cdots + \Psi_{n/u}(t) \cdot t^{n/u - 1} = \Psi_n(t)$, so that $\operatorname{RRes}_{n/u}(\Psi_n(t)) = 0$. If $n$ has no proper, non-trivial divisor, then $n= 1$ or a prime, in which case $\Psi_n(t) = \Phi_n(t)$. So we are in the previous case. \end{proof}

\begin{proposition}[$\discr{\Phi_n(t)} \cdot \prodant{\Phi_n(t)}$] \label{ExampleCycloDiscrProd} Consider a natural number $n > 1$ and the corresponding cyclotomic polynomial $\Phi_n(t)$. Then $\discr{\Phi_n(t)} \cdot \prodant{\Phi_n(t)}$ divides a natural power of $n$. \end{proposition}

\begin{proof} If $n = 2$, then $\Phi_2(t) = 1 + t$, so that we are in the linear case again. So we assume that $n > 2$. Since the cyclotomic field $K_n$ corresponding with $\Phi_n(t)$ is monogenic, we know that $\operatorname{Discr}({\Phi_n(t)})$ coincides with the field discriminant 
	$ \Delta_{K_n} = (-1)^{\varphi(n)/2} \cdot n^{\varphi(n)} / {  \prod_{\substack{p \in \mathbb{P}\\ p | n}  } } p^{\varphi(n) / \varphi(p)}$ of $K_n$. We see, in particular, that $\discr{\Phi_n(t)}$ divides $n^{n}$. Now set $r(t) := \Phi_n(t)$. We use the notation of \ref{SubSecInvarsDiscrProd}. By construction, for each root $\lambda$ of the monic polynomial $\overline{r}(t)$, there exists a natural $m$ (properly) dividing $n$, such that $\lambda$ is an $m$'th root of unity. Since also $\overline{r}(t) \in \Z[t]$, there exist non-negative integers $a_m$ such that $\overline{r}(t) = \prod_{\substack{m | n\\m \neq n}} \Phi_m(t)^{a_m}. $ So $\prodant{r(t)} =  \operatorname{Res}( r(t) , \overline{r}(t)) = \prod_{\substack{m | n\\m \neq n}} \operatorname{Res}(\Phi_n(t),\Phi_m(t))^{a_m}. $ The factors in this expression were computed explicitly by E. Lehmer \cite{LehmerCyclo}, Apostol \cite{ApostolCyclo}, Dresden \cite{DresdenCyclo}, and several others \cite{ChengMcKayWang}. We see, in particular, that also $\prodant{\Phi_n(t)}$ divides a natural power of $n$. This finishes the proof.
\end{proof}

We have, for example: $\discr{\Phi_6(t)} \cdot \prodant{\Phi_6(t)} = (3) \cdot (2^2)$ and $\discr{\Phi_{15}(t)} \cdot \prodant{\Phi_{15}(t)} = (3^4 \cdot 5^6) \cdot (3^{24} \cdot 5^8)$.

\begin{proposition}[$\discr{\Psi_n(t)} \cdot \prodant{\Psi_n(t)}$] \label{ExSplitPolyDP} Consider a natural number $n > 1$ and the corresponding split polynomial $\Psi_n(t) = 1 +  t + \cdots + t^{n-1}$. Then
	$\discr{\Psi_n(t)} \cdot \prodant{\Psi_n(t)} = \left( n^{n-2} \right) \cdot \left( n^{n-1} \right) .$ \end{proposition}

\begin{proof} We may assume that $n > 2$. According to formula  \eqref{FormulaExplicitDiscrSylvester}, we have:	
	\begin{eqnarray*}
		\discr{\Psi_n(t)} &=& \det \left( \operatorname{Syl}(\Psi_n(t),\Psi_n'(t)) \right) \\
		&=& \operatorname{Res}(\Psi_n(t),\Psi_n'(t)) \\
		&=& \operatorname{Res}(t-1,\Psi_n'(t))^{-1} \cdot \operatorname{Res}(t^n-1,\Psi_n'(t)) \\
		&=& (-1)^n \cdot \binom{n}{2}^{-1} \cdot \det \left( \operatorname{Syl} (t^n-1 , \Psi_n'(t)) \right).
	\end{eqnarray*}
	By performing row and column operations on the matrix $\operatorname{Syl}(t^n-1,\Psi_n'(t))$, we obtain the circulant determinant $  \det \left( \operatorname{Syl}(t^n-1,\Psi_n'(t)) \right) = - \det \left( \operatorname{Circ}(0,1,2,\ldots,n-1) \right) = (-1)^n \cdot \binom{n}{2} \cdot n^{n-2}$. So $\discr{(t)} = n^{n-2}$. Now set $r(t) := \Phi_n(t)$. We use the notation of \ref{SubSecInvarsDiscrProd} again. Let $\omega_n$ be a primitive $n$'th root of unity. We then have $ \overline{r}(t) = \prod_{\substack{0 < i , j < n\\i+j \equiv 0 \mod n}} (t-\omega_n^{i+j}) = (t-1)^{n-1},$ and therefore $\prodant{r(t)} = \operatorname{Res}(r(t),\overline{r}(t)) = \Psi_n(1)^{n-1} = n^{n-1}.$ \end{proof}

\section{Structure theorems for Lie rings}
\label{SectionLieRings}\label{SectionLieRingsAux}

\subsection{Constructing an embedding of Lie rings}

\begin{lemma}[Binomial commutator formula] \label{LemmaBinomialCommutatorFormula} Consider a Lie ring $L$ with coefficients in a commutative ring $R$. Let $\map{\gamma}{L}{L}$ be a Lie endomorphism of $L$, let $\lambda,\mu \in R$ be coefficients, and let $v,w \in L$. For all $m \in \N$, we then have 
	$$ ({\gamma} - (\lambda \cdot \mu) \cdot \id_{L})^m([v,w]) =  \sum_{0 \leq i \leq m} \binom{m}{i} \cdot [\lambda^{m-i} \cdot ({\gamma} - \lambda \cdot \id_L)^i(v),{\gamma}^i \circ ({\gamma} - \mu \cdot \id_L)^{m-i} (w)] .$$ If, for some $m_\lambda,m_\mu \in \N$, we have $(\gamma - \lambda \cdot \id_L)^{m_\lambda}(v) = 0_L = (\gamma - \mu \cdot \id_L)^{m_\mu}(w),$ then we also have $(\gamma - \lambda \cdot \mu \cdot \id_L)^{m_\lambda + m_\mu }([v,w]) = 0_L.$
\end{lemma}

\begin{proof}
	One can prove the first formula by a simple induction on $m$ and Pascal's binomial identity $\binom{n}{i} = \binom{n-1}{i-1} + \binom{n-i}{i}$. By specializing to $m = m_\lambda + m_\mu$, we also obtain the second formula.
\end{proof}

For an integer $h$ and a Lie ring $M$ with coefficients in a ring $R$, we define the $h$-torsion ideal $T_h(M)$ of $M$ by $ T_h(M) := \{ v \in M | \exists n \in \N : h^n \cdot v = 0_M \} .$ It is clear that such a set $T_h(M)$ is a Lie ideal of $M$ that is invariant under $R$-multiplications and $M$-endomorphisms.

\begin{theorem}\label{TheoremEmbedding}
	Consider a Lie ring $L$, together with an endomorophism $\map{\gamma}{L}{L}$, and a polynomial $r(t) \in \Z[t]$ such that $r(\gamma) = 0_L$. Suppose that $L$ has no $(\discr{r(t)} \cdot \prodant{r(t)})$-torsion. Then the Lie ring $\discr{r(t)} \cdot L$ can be embedded into a Lie ring $\overline{K}$ that is graded by the semi-group $(\overline{\Q},\cdot)$ and supported by the root set $X$ of $r(t)$ in $\overline{\Q}$.
\end{theorem}

Let us use the abbreviations $\delta := \discr{r(t)}$ and $\pi := \prodant{r(t)}$. We may assume that $r(t)$ is not a constant polynomial, since otherwise we trivially have $\delta \cdot L = 0_L$, and there is nothing to prove. Let $\lambda_1,\ldots,\lambda_l$ be the distinct roots with respective multiplicities $m_1,\ldots,m_l$ and let $a$ be the leading coefficient, so that $$r(t) = a \cdot \prod_{1 \leq i \leq l}(t - \lambda_i)^{m_i}.$$ Let $R := \Z[\lambda_1,\ldots,\lambda_l]$ be the ring generated by the roots and let $\F$ be the field of fractions of $R$. We next introduce a new Lie ring $\widetilde{L} := R \otimes_{\Z} L$ with coefficients in $R$. This new Lie ring $\widetilde{L}$ naturally admits the Lie ring endomorphism $\mapl{\widetilde{\gamma}}{\widetilde{L}}{\widetilde{L}}{\sum_j a_j \otimes v_j}{\sum_j a_j \otimes (\gamma(v_j))}$ and this map inherits the property $r(\widetilde{\gamma}) = 0_{\widetilde{L}}$. We further define the ideal $T := T_{a \cdot \delta \cdot \pi}(\widetilde{L}) = T_{\delta \cdot \pi}(\widetilde{L})$ of $\widetilde{L}$. And for each $\lambda \in R$, we define the $R$-submodule of $\widetilde{L}$: \begin{eqnarray*}
	E_{\lambda} &:=& \{ v \in \widetilde{L} \vert \exists n \in \N : (\widetilde{\gamma} - \lambda \cdot \id_{\widetilde{L}})^n v \in T  \} \\
	&=& \{ v \in \widetilde{L} \vert \exists n_1,n_2 \in \N : (\delta \cdot \pi)^{n_1} \cdot (\widetilde{\gamma} - \lambda \cdot \id_{\widetilde{L}})^{n_2} v = 0_{\widetilde{L}}  \}.
\end{eqnarray*}

\textbf{Claim $1$:} \emph{For all $\lambda,\mu \in R$, we have}
\begin{enumerate}
	\item[a.]   $T \subseteq \bigcap_{\nu \in R} E_\nu$.
	\item[b.] $[E_{\lambda} , E_{\mu}] \subseteq E_{\lambda \cdot \mu}.$
	\item[c.] \emph{If $r(\lambda) = 0 = r(\mu)$ and $r(\lambda \cdot \mu) \neq 0$, then $[E_\lambda , E_\mu]  \subseteq T$.}
\end{enumerate}

\begin{proof} \let\qed\relax $(a.)$ This property holds trivially.  \end{proof}

\begin{proof} \let\qed\relax	$(b.)$ Select arbitrary $v \in E_{\lambda}$ and $w \in E_{\mu}$. By definition, there exist $n_1,n_2,n_3,n_4 \in \N$ such that $ ({\delta} \cdot {\pi})^{n_1} \cdot (\widetilde{\gamma} - \lambda \cdot \id_{\widetilde{L}})^{n_2} v = 0_{\widetilde{L}} = ({\delta} \cdot {\pi})^{n_3} (\widehat{\gamma} - \mu \cdot \id_{\widetilde{L}})^{n_4} w .$ By using the commutator formula of lemma \ref{LemmaBinomialCommutatorFormula}, we get: $ ({\delta} \cdot {\pi})^{n_1 + n_3} \cdot (\widetilde{\gamma} - \lambda \cdot \mu \cdot \id_{\widetilde{L}})^{n_2 + n_4} [v,w] = 0_{\widetilde{L}}.$ \end{proof}

\begin{proof} \let\qed\relax
	$(c.)$ In view of the above, it suffices to show that $E_{\lambda \cdot \mu} \subseteq T$. Let us abbreviate $\nu := \lambda \cdot \mu$. Let $r(t)$ be given by $r(t) := \sum_{0 \leq i \leq d} a_i \cdot t^i \in \Z[t]$. Then $r(t)-r(\nu) = s(t) \cdot (t-\nu) ,$ where $s(t) = \sum_{0 \leq i \leq d} a_i \cdot  ( \sum_{j+k = i} t^j \cdot \nu^k) \in R[t]$. 	Select an arbitrary $v \in E_{\nu}$. By definition, there exist $n_1,n_2 \in \N$ such that $(a \cdot {\delta} \cdot {\pi})^{n_1} \cdot (\widetilde{\gamma} - \nu \cdot \id_{\widetilde{L}})^{n_2} v = 0_{\widetilde{L}}$. Then also \begin{equation}
	({\delta} \cdot {\pi})^{n_1} \cdot (r(\widetilde{\gamma}) - r(\nu) \cdot \id_{\widetilde{L}})^{n_2} v = s(\widetilde{\gamma})^{n_2} \left( ({\delta} \cdot {\pi})^{n_1} \cdot (\widetilde{\gamma} - \nu \cdot \id_{\widetilde{L}})^{n_2} v \right) = 0_{\widetilde{L}}. \label{Taylor1}
	\end{equation}
	But, since $r(\widetilde{\gamma})$ vanishes on $\widetilde{L}$, we also have: \begin{equation}
	(r(\widetilde{\gamma}) - r(\nu) \cdot \id_{\widetilde{L}})^{n_1} v = \sum_{j+{n_2} = {n_1}} \binom{{n_1}}{j,{n_2}} \cdot r(\widetilde{\gamma})^j \left(  (-r(\nu))^{n_2} \cdot v \right) = (-r(\nu))^{n_1} \cdot v . \label{Taylor2}
	\end{equation} 
	By combining \eqref{Taylor1} and \eqref{Taylor2}, we obtain $({\delta} \cdot {\pi})^{n_2} \cdot r(\nu)^{n_1} \cdot v = 0_{\widetilde{L}}$. Since $r(\lambda) = 0 = r(\mu)$ and $r(\lambda \cdot \mu) \neq 0$, there exists a ${n_3} \in \N$ such that $({\delta} \cdot {\pi})^{{n_2}+{n_3}} \cdot v = 0_{\widetilde{L}}$. So $v \in T$. 
\end{proof}

\textbf{Claim $2$:} \emph{The $R$-submodule $K := \sum_{ \lambda \in X } E_\lambda$ is a Lie subring of $\widetilde{L}$.}

\begin{proof} \let\qed\relax
	According to the Jacobi-identity and the bi-linearity of the Lie bracket, the Lie $R$-subalgebra of $\widetilde{L}$ generated by the $R$-submodule $K$ of $\widetilde{L}$ is the $R$-span of left-normed words $w$ of the form $w := [v_1,\ldots,v_n],$ where each $v_i $ is contained in some $ K_{\mu_i}$. So it suffices to show that for such a word $w$, we have $w \in K$. Let us do this. We may suppose that $r(\mu_1) = \cdots = r(\mu_n) = 0$, since otherwise $w$ is contained in the ideal $T$ and therefore in $K$. If $\mu_1 , \mu_1 \cdot \mu_2 , \ldots , \mu_1 \cdots \mu_n := \lambda$ are all roots of $r(t)$, then we need only apply claim $1.b$ $(n-1)$-times in order to conclude that $w \in K_\lambda$ and therefore $w \in K$. Else, there exists an index $n_0 \in \{ 1,\ldots ,n-1\}$, such that $\mu_1 \cdots \mu_{n_0} := \mu$ is a root, but $\mu \cdot \mu_{n_0+1}$ is not. Define $u := [v_1,\ldots,v_{n_0}]$. By applying claim $1.b$ $(n_0-1)$-times, we see that $u \in K_\mu$. By applying claim $1.c$, we see that $[u,v_{n_0+1}] \in T$. So also $w = [[u,v_{n_0+1}],v_{n_0+2},\ldots,v_n] \in T \subseteq K$.
\end{proof}

We now note that $T = \{ v \in K | \exists n \in \N : (\delta \cdot \pi)^{n_1} \cdot v = 0_K \}$. So we consider the quotient $$\overline{K} := K/T = (\sum_{\lambda \in X} E_\lambda)/T = \sum_{ \lambda \in X } (E_\lambda / (E_\lambda \cap T)) = \sum_{ \lambda \in X } (E_\lambda / T).$$ For each $\lambda \in \overline{\Q}$, we define the $R$-submodule $\overline{K}_\lambda $ of $\overline{K}$ by: 
$$
\overline{K}_{\lambda} := \begin{cases}
E_\lambda / T \text{ if } \lambda \in X, \\ 
T / T \text{ if } \lambda \not \in X.
\end{cases}
$$

\textbf{Claim $3$:} $\overline{K} = \bigoplus_{\lambda \in \overline{\Q}}\overline{K}_\lambda$ is a grading of the Lie ring $\overline{K}$ by $(\overline{\Q}^\times,\cdot)$ and its support is contained in $X$.

\begin{proof} \let\qed\relax
	In view of claims $1$ and $2$, we need only show that the decomposition $\overline{K} = \sum_{\lambda \in \overline{\Q}} \overline{K}_{\lambda}$ of $\overline{K}$ into $R$-submodules is direct. We select arbitrary $v_{\lambda_1} \in {K}_{\lambda_1} ,\ldots,v_{\lambda_l} \in {K}_{\lambda_l} $ such that $v_{\lambda_1} + \cdots + v_{\lambda_l} \in T$ and we then need to show that $v_{\lambda_1} , \ldots , v_{\lambda_l} \in T$. So we select an arbitrary $i \in \{ 1,\ldots,l\}$ and will show that $v_{\lambda_i} \in T$. 	
	By definition, there exist ${n_1},{n_2} \in \N$ such that for all $j \in \{ 1,\ldots,l \}$, we have:  \begin{equation}
	({\delta} \cdot {\pi})^{{n_2}} \cdot (\widetilde{\gamma} - \lambda_j \cdot \id_K)^{n_1} v_{\lambda_j} = 0_K \label{EqVanish1}
	\end{equation} Define the auxiliary polynomial $s(t) = \prod_{j \neq i}(t-\lambda_j)^{{n_1}} \in R[t]$. Then the theory of resultants tells us that there exist polynomials $g(t),h(t) \in R[t]$ such that \begin{equation} g(t) \cdot s(t) + h(t) \cdot (t - \lambda_i)^{{n_1}} = \operatorname{Res}(s(t),(t-\lambda_i)^{{n_1}}) = s(\lambda_i)^{{n_1}}. \label{EqResultantWeakGrading} \end{equation} Since $T$ is invariant under $\widetilde{\gamma}$ and multiplication by elements of $R$, we see that $({\delta} \cdot {\pi})^{{n_2}} \cdot s(\widetilde{\gamma}) (v_{\lambda_i}) = ({\delta} \cdot {\pi})^{{n_2}} \cdot s(\widetilde{\gamma}) (v_{\lambda_1} + \cdots + v_{\lambda_l}) \in T.$ So, by definition, there exists a ${n_3} \in \N$ such that \begin{equation}
	({\delta} \cdot {\pi})^{{n_2}+{n_3}} \cdot s(\widetilde{\gamma}) (v_{\lambda_i}) = 0_K \label{EqVanish2}.
	\end{equation} By first evaluating \eqref{EqResultantWeakGrading} in $\widetilde{\gamma}$ and then in $v_{\lambda_i}$, and by substituting \eqref{EqVanish1} and \eqref{EqVanish2}, we obtain:
	\begin{eqnarray*}
		(\delta \cdot \pi)^{{n_2}+{n_3}} \cdot s(\lambda_i)^{n_1} v_{\lambda_i} &=& (\delta \cdot \pi)^{{n_2}+{n_3}} \cdot g(\widetilde{\gamma}) \left( s(\widetilde{\gamma} ) v_{\lambda_i} \right)   \\
		& & \phantom{ooo} + (\delta \cdot \pi)^{{n_2}+{n_3}} \cdot h(\widetilde{\gamma}) \left( (\widetilde{\gamma}-\lambda_i \cdot \id_{K})^{n_1} v_{\lambda_i} \right) \\
		&=& 0_K.
	\end{eqnarray*}
	Since the factor $s(\lambda_i)$ divides ${\delta}$ in the ring $R$, there exists some ${n_4} \in \N$ such that also $({\delta} \cdot {\pi})^{{n_2}+{n_3}+{n_2} \cdot {n_4}} \cdot v_{\lambda_i} = 0_K$. We may therefore conclude that $v_{\lambda_i} \in T$.		
\end{proof}

\textbf{Claim $4$:} \emph{We have the inclusion $1 \otimes ({\delta} \cdot L) \subseteq K$.}

\begin{proof}  \let\qed\relax
	For each $i \in \{ 1,\ldots,l \}$ and each $j \in \{ 0 , \ldots, m_i-1\}$, we define the polynomial $P_{i,j}(t) := r(t) / (t - \lambda_i)^{m_i-j} $ with coefficients in the ring $R := \Z[\lambda_1,\ldots,\lambda_l]$. Let us first show that, for each polynomial $P_{i,j}(t)$, there exists a coefficient $\theta_{i,j} \in R$ such that \begin{eqnarray}
	\operatorname{Discr}_\ast(r(t)) = \sum_{1 \leq i \leq l} \sum_{0 \leq j \leq m_i -1} \theta_{i,j} \cdot P_{i,j}(t) . \label{DiscrEquation}\end{eqnarray}	
	The partial fraction decomposition of $a / r(t)$ is given by $$\frac{a}{r(t)} = \sum_{1 \leq i \leq l} \sum_{0 \leq j \leq m_i-1} \frac{1}{j!} \cdot \left( \left( \frac{a}{P_{i,0}(t)} \right)^{(j)}(\lambda_i) \right) \cdot (t - \lambda_i)^{j-m_i}.$$ For each $i \in  \{1, \ldots, l\}$, we define the auxiliary polynomial $s_i(t) := \prod_{\substack{1 \leq j \leq l\\j\neq i}} (t-\lambda_j) \in R[t]$. Then the $j$'th derivative of $a/(P_{i,0}(t))$ is clearly of the form $b_{i,j}(t) / (s_i(t))^{2m}$, for some explicitly computable $b_{i,j}(t) \in R[t]$. We see, in particular, that $$a = \sum_{1 \leq i \leq l} \sum_{0 \leq j \leq m_i-1} \frac{b_{i,j}(\lambda_i)}{j! \cdot (s_i(\lambda_i))^{2m}} \cdot P_{i,j}(t).$$ 
	After multiplying both sides of this equality by  $\operatorname{Discr}_\ast(r(t)) / a$, we see that $$\theta_{i,j} := \frac{(m-1)!}{j!} \cdot \left( (-1)^{m(l-1)} \cdot a^{2d^2} \cdot \prod_{\substack{1 \leq k , n \leq l\\i \neq k \neq n \neq i}} (\lambda_k - \lambda_n)^m \right) \cdot b_{i,j}(\lambda_i) $$ is a solution to \eqref{DiscrEquation} in the ring of coefficients $R$, where $m := \max \{ m_1 , \ldots , m_l \}$. \\
	
	We now select an arbitrary $v \in L$. Corresponding with the $(i,j)$'th term of \ref{DiscrEquation}, we define $v_{i,j} := P_{i,j}(\widetilde{\gamma}) (1 \otimes v) \in \widetilde{L}$. Then we observe that $(\widetilde{\gamma} - \lambda_i \cdot \id_K)^{m_i-j} v_{i,j} = r(\widetilde{\gamma}) (1 \otimes v) = 0_K$, so that $v_{i,j} \in K_{\lambda_i}$. By evaluating the expression \eqref{DiscrEquation} in $\widetilde{\gamma}$ and then in $1 \otimes v$, we see that $1 \otimes ({\delta} \cdot v) = \sum_{1 \leq i \leq l} \sum_{0 \leq j \leq m_i-1} \theta_{i,j} \cdot v_{i,j} \in \sum_{1 \leq i \leq l} \sum_{0 \leq j \leq m_i-1} R \cdot K_{\lambda_i} \subseteq K.$ So we may indeed conclude that $1 \otimes ({\delta} \cdot L) \subseteq K$.
\end{proof}

\textbf{Claim $5$:} The Lie ring $\delta \cdot L$ embeds into the Lie ring $\overline{K} = K/T$. 

\begin{proof}
	We consider the embedding $\mapl{E}{\delta \cdot L}{K}{\delta \cdot v}{1 \otimes (\delta \cdot v)}$ and the projection $\mapl{P}{K}{K/T}{w}{w \operatorname{mod} T}$. Since $L$ has no $(\delta \cdot \pi)$-torsion, neither has $\delta \cdot L$. So the composition $\map{P \circ E}{\delta \cdot L}{K/T}$ is the required embedding of Lie rings.
\end{proof}

\subsection{Bounded nilpotency of graded Lie rings}

\begin{theorem} \label{TheoremLieAFold} Let $\mathbb{F}$ be a field. If a Lie ring ${K}$ is graded by $(\mathbb{F}^\times,\cdot)$ with finite, arithmetically-free support $X$, then ${K}$ is nilpotent of class at most $ {|X|}^{2^{|X|}} .$
\end{theorem}

We now fix a finite, arithmetically-free subset $X$ of the multiplicative group $(\mathbb{F}^\times,\cdot)$ of a field $\mathbb{F}$. We want to show that every Lie ring $L$ that is graded by $\mathbb{F}^\times$ and supported by $X$ is nilpotent of class at most $|X|^{2^{|X|}}$. 
If $|X| = 0$, then $L = 0$, and we indeed have $\operatorname{c}(L) = 0 \leq 0^{2^0}$. If $|X| = 1$, then $[L,L] \subseteq L_{G \setminus X} = \{ 0_L \}$, so that we have $\operatorname{c}(L) = 1 \leq 1^{2^1}$. We may therefore assume from now on that  $x := |X| \geq 2$. 
We begin by proving a useful property of this arithmetically-free set $X$.

\begin{lemma}[Growth or progression] \label{GroPro} Consider an abelian group $(F,\cdot)$. For every $k \in \N$, and elements $b_1,\ldots,b_{k}$ of $F \setminus \{ 1_F \}$, we can define the set $S := S(b_1,\ldots,b_{k}) := \bigcup_{0 \leq i \leq k}\{ b_{\sigma(1)}  \cdots b_{\sigma(i)} | \sigma \in \operatorname{Sym}(k) \}$ of partial products. Then: $|S| \geq k+1$, or
	$S$ contains a progression $g , g \cdot c , g \cdot c^2, \ldots, g \cdot c^{k}$ with $c \in \{b_1,\ldots,b_k\}$.
\end{lemma}

\begin{proof}
	Let us prove the lemma by induction on $k \in \N$. If $k = 1$, then the statement is clearly true. So we may suppose that $k > 1$. We suppose that the second conclusion fails for $S$ and we will prove that $|S| \geq k+1$. Define $S' := S(b_1,\ldots,b_{k-1})$ and $S := S(b_1,\ldots,b_k)$. Then $S = S' \cup (S' \cdot b_k)$ and the second conclusion also fails for $S'$. The induction hypothesis now tells us that $|S'| \geq k$. If $S = S' \cdot b_k$, then $S = S' \cdot b_k = S' \cdot b_k^2 = \cdots = S' \cdot b_k^k$. But then $S'$ would satisfy condition $2$ (contradicting our assumption). So $S' \cdot b_k \neq S$, and therefore $|S| = |S' \cup (S' \cdot b_k)| \geq 1 + |S'| \geq 1+k$.
\end{proof}

\begin{proposition}[Escape] \label{Escape} For all elements $a,b_1,\ldots,b_{|X|}$ of $X$, there exists a permutation $\sigma \in \operatorname{Sym}(|X|)$ and a cut-off $k \in \{1,\ldots,|X|\}$ such that $a \cdot b_{\sigma(1)} \cdots b_{\sigma(k)} \in G \setminus X.$
\end{proposition}

\begin{proof}
	We suppose that the conclusion is false, and we will derive a contradiction. Consider $S := S(b_1,\ldots,b_{|X|})$. Then $a \cdot S \subseteq X$ and $|S| = |a \cdot S| \leq |X|$. Lemma \ref{GroPro} now implies that $S$ contains a progression $b,b\cdot c, b \cdot c^2,\ldots,b \cdot c^{|X|}$ with $c \in X$. Then  $a \cdot S$, and therefore $X$, contains the progression $a \cdot b,(a \cdot b) \cdot c,(a \cdot b) \cdot c^2,\ldots,(a \cdot b) \cdot c^{|X|}$ with $c \in X$. We conclude that $X$ is not arithmetically-free. This is a contradiction.
\end{proof}

We will break down our proof of theorem \ref{TheoremLieAFold} into several steps (that are directly inspired by the work of Kreknin and Kostrikin (cf. E. Khukhro's book \cite{KhukhroNilpotentGroupsAutomorphisms})).

\begin{lemma} \label{PermContr}Consider a Lie ring $L$. For every permutation $\sigma \in \operatorname{Sym}(X)$ and for all elements $v,w_1,\ldots,w_x \in L$, we have $ [v,w_1,\ldots,w_x] \in [v,w_{\sigma(1)},\ldots,w_{\sigma(x)}] + \operatorname{id}_L([v,[L,L]]) .$
\end{lemma}

\begin{proof}
	A repeated application of the Jacobi-identity shows that the difference vector $[v,w_1,\ldots,w_x] - [v,w_{\sigma(1)},\ldots,w_{\sigma(x)}]$ is in the $\Z$-linear span of $[v,[L,L]]$, $[v,[L,L],L]$,  $\ldots,[v,[L,L],\underbrace{L,\ldots,L}_\text{$x$ times}],$ and therefore in the ideal of $L$ that is generated by $[v,[L,L]]$.
\end{proof}

\begin{proposition} \label{PropGroup}
	Consider a Lie ring $L$ that is graded by $(\mathbb{F}^\times,\cdot)$ and supported by $X$. For every homogeneous ideal $I$ of $L$, we have $[I,\underbrace{L,\ldots,L}_\text{$x$ times}] \subseteq [I,[L,L]].$
\end{proposition}

\begin{proof}
	Let the gradings be given by $L = \bigoplus_{\lambda \in \mathbb{F}^\times} L_\lambda$ and $I = \bigoplus_{\lambda \in \mathbb{F}^\times} I_\lambda$. Since the Lie-bracket is bi-linear, it suffices to show that every left-normed bracket of the form $u := [v_a,w_{b_1} , \ldots , w_{b_x}],$ with $v_a \in I_a$ and $w_{b_1} \in L_{b_1} , \ldots , v_{b_x} \in L_{b_x}$, is contained in the ideal $[I,[L,L]]$. If any of the $a,b_1,\ldots,b_x$ are in $\mathbb{F}^\times \setminus X$, then we indeed have $u = 0 \in [I,[L,L]]$. So we may assume that $a,b_1,\ldots,b_x \in X$. According to proposition \ref{Escape}, there exists a permutation $\sigma \in \operatorname{Sym}(X)$ and a cut-off $k \in \{ 1 , \ldots , x \}$, such that $c := a \cdot b_{\sigma(1)} \cdots b_{\sigma(k)} \in \mathbb{F}^\times \setminus X$, and therefore $L_c = \{ 0_L \}$. We then apply lemma \ref{PermContr} to the bracket $u$ and $\sigma$ in order to obtain $u \in [[v_a,w_{\sigma(1)} , \ldots , w_{\sigma(k)}] , w_{\sigma(k+1)} , \ldots , w_{\sigma(x)}] + [I,[L,L]].$ The grading property implies that the sub-bracket $[v_a,w_{\sigma(1)} , \ldots , w_{\sigma(k)}]$ is contained in the homogeneous component $L_{c} = \{ 0_L \}$ of $L$. So $u \in [0,w_{\sigma(k+1),\ldots,w_{\sigma(x)}}] +  [I,[L,L]] = [I,[L,L]]$.
\end{proof}

\begin{proposition} \label{PropGroupMult}
	Consider a Lie ring $L$ that is graded by $(\mathbb{F}^\times,\cdot)$ and supported by $X$. For every homogeneous ideal $I$ of $L$ and $l \in \N$, we have $[I,\underbrace{L,\ldots,L}_\text{$l x$ times}] \subseteq [I,\underbrace{[L,L],\ldots,[L,L]}_\text{$l$ times}].$
\end{proposition}

\begin{proof}
	We proceed by induction on $l \in \N$. The base case $l=1$ corresponds with the previous proposition. So we may assume that $l > 1$. Then $J := [I,\underbrace{L,\ldots,L}_\text{$(l-1) x$ times}]$ is a homogeneous ideal of $L$ and proposition \ref{PropGroup} gives us the inclusion $[J,\underbrace{L,\ldots,L}_\text{$x$ times}] \subseteq [J,[L,L]].$ The induction hypothesis for $l-1$ also gives the inclusion $J \subseteq [I,\underbrace{[L,L],\ldots,[L,L]}_\text{$(l-1)$ times}]$. These two combine to give $[I,\underbrace{L,\ldots,L}_\text{$l x$ times}] \subseteq [[I,\underbrace{[L,L],\ldots,[L,L]}_\text{$(l-1)$ times}],[L,L]] = [I,\underbrace{[L,L],\ldots,[L,L]}_\text{$l$ times}].$
\end{proof}

\begin{proposition} \label{NestedIdeals} Consider a Lie ring $L$ that is graded by $(\mathbb{F}^\times,\cdot)$ and supported by $X$. For each $l \in \N$, we have the inclusion of ideals:
	$ \Gamma_{2+l x}(L) \subseteq \Gamma_{l+1}(\Delta_{1}(L)) .$
\end{proposition}

\begin{proof}
	We need only specialize the previous proposition to the homogeneous ideal $I := [L,L]$ of $L$.
\end{proof}

\begin{proposition} Let $s \in \N$. For every $\mathbb{F}^\times$-graded Lie ring $L$ that is supported by $X$, we have the inclusion of ideals 
	$\Gamma_{1 + (x^s-1)/(x-1)}(L) \subseteq \Delta_s(L).$
\end{proposition}

\begin{proof}
	We proceed by induction on $s$. If $s = 1$, then the statement is true by definition: $\Gamma_2(L) = [L,L] = \Delta_1(L) $. Now suppose that $s > 1$. We then define $l := (x^{s-1}-1)/(x-1)$ and we observe that $(x^s-1)/(x-1) = 1 + x l$. According to proposition \ref{NestedIdeals}, we then have $\Gamma_{2 + x l}(L) \subseteq \Gamma_{1+l}(\Delta_1(L))$. We note that the Lie ring $\Delta_1(L)$ is naturally graded by $(\mathbb{F}^\times,\cdot)$ and supported by $X$. So the induction hypothesis for $s-1$ tells us that $\Gamma_{l+1}(\Delta_1(L)) \subseteq \Delta_{s-1}(\Delta_1(L)) = \Delta_s(L)$. We therefore obtain $\Gamma_{2 + x l}(L) \subseteq \Gamma_{1+l}(\Delta_1(L)) \subseteq \Delta_s(L).$
\end{proof}

\begin{corollary} If a Lie ring $L$ is graded by $(\mathbb{F}^\times,\cdot)$ and supported by $X$, then it is nilpotent of class 
	$\operatorname{c}(L) \leq {|X|}^{2^{|X|}}.$
\end{corollary}

\begin{proof}
	We may suppose that $x \geq 2$. Since $X$ is an arithmetically-free subset of $\mathbb{F}^\times$, it does not contain $1$. So, according to Shalev's generalization of Kreknin's theorem in \cite{ShalevAutomorphismsFiniteRank}, the Lie ring $L$ is solvable of derived length $\operatorname{dl}(L) \leq 2^{x} =: s$. Proposition \ref{NestedIdeals} now gives $\Gamma_{{x}^{2^{x}}+1}(L)\subseteq \Gamma_{1 + (x^s-1)/(x-1)}(L) \subseteq \Delta_{s}(L) \subseteq \Delta_{\operatorname{dl}(L)}(L) = \{ 0_L\}.$
\end{proof}

This finishes the proof of theorem \ref{TheoremLieAFold}.

\subsection{Proof of theorem \ref{IntroTheoremLie}}

\begin{proof}
	We first use theorem \ref{TheoremEmbedding} to embed $\discr{r(t)} \cdot L$ into a Lie ring $K$ that is graded by $(\overline{\Q}^\times,\cdot)$ and supported by the root set $X$ of $r(t)$. According to corollary \ref{CorollaryGoodImpliesAF}, $X$ is an arithmetically-free subset of $(\overline{\Q}^\times,\cdot)$. So theorem \ref{TheoremLieAFold} implies that $\discr{r(t)} \cdot L$ is nilpotent of class $\operatorname{c}(\discr{r(t)} \cdot L) \leq d^{2^d}$. Since $(L,+)$ has no $\discr{r(t)}$-torsion, also $L$ is nilpotent of class at most $d^{2^d}$.
\end{proof}

\section{Structure theorems for groups} \label{SectionStructureForGroups}

\subsection{Proof of theorem \ref{MainTheoremA}} \label{SubsecMainThA}

\paragraph{Preliminaries.} We recall some basic terminology. Let $G$ be a group and let $\alpha$ be one of its automorphisms.  
A subgroup $H$ of $G$ is \emph{$\alpha$-invariant} if $\alpha(H) \subseteq H$. An $\alpha$-invariant \emph{section} of $G$ is a quotient $A/B$ of a $\alpha$-invariant subgroup $A$ of $G$ by a $\alpha$-invariant, normal subgroup $B$ of $A$. A characteristic section of $G$ is a quotient $A/B$ of a $G$-characteristic group $A$ by a $G$-characteristic subgroup $B$ of $A$. A section $A/B$ of $G$ is \emph{proper} if $|A/B| < |G|$. The following result is well-known.

\begin{lemma} \label{LemmaRegularPrelims} Let $G$ be a finite group and let $\alpha$ be a fix-point-free automorphism. $(i.)$ The map $\mapl{\tau}{G}{G}{x}{x^{-1} \cdot \alpha(x)}$ is a bijection. $(ii.)$ If $A/B$ is an $\alpha$-invariant section then the corresponding automorphism $\mapl{\overline{\alpha}}{A/B}{A/B}{a \cdot B}{\alpha(a) \cdot B}$ is also fix-point-free. $(iii.)$ If $p$ is a prime, then $G$ has a $p$-Sylow subgroup that is $\alpha$-invariant. $(iv.)$ If $G$ is solvable and $H$ is a Hall-subgroup of $G$, then some conjugate of $H$ is $\alpha$-invariant. \end{lemma}

Let $m \in \Z$. A group $H$ is said to be an $m'$-group if $H$ has no $m$-torsion (i.e.: only the trivial element $h \in H$ satisfies $h^m = 1_H$).

\subsubsection{Solvable case}

\begin{theorem} \label{CorollaryDecompositionKN}
	Consider a finite, solvable group $G$, together with a fix-point-free automorphism $\map{\alpha}{G}{G}$, and an identity $r(t)$ of the automorphism. Then $G$ has a characteristic $\tcn{r(t)}$-subgroup $S$ with nilpotent $\tcn{r(t)}'$-quotient $G/S$.
\end{theorem}

We suppose that the statement is false and we will then deduce a contradiction. We may well suppose that $(G,\alpha,r(t))$ is a counter-example of minimal order $|G|$. We then see that $G$ is not nilpotent and not a $\tcn{r(t)}$-group. We also see that all proper, $\alpha$-invariant sections of $G$ satisfy the conclusion of the theorem. We will use these observations repeatedly but not always explicitly. \\

\textbf{Claim $1$:} \emph{If a characteristic subgroup $K$ of $G$ is a $\tcn{r(t)}$-group, then $K = \{1_G\}$.}

\begin{proof}  \let\qed\relax
	Suppose that the characteristic subgroup $K$ is non-trivial. Then the proper, $\alpha$-invariant section $G/K$ of $G$ has a characteristic, $\tcn{r(t)}$-subgroup $S/K$ with nilpotent $\tcn{r(t)}'$-quotient $(G/K)/(S/K) \cong G/S$. But then $S$ is a characteristic $\tcn{r(t)}$-subgroup of $G$ with nilpotent $\tcn{r(t)}'$-quotient, contradicting our initial assumption. 
\end{proof}

\textbf{Claim $2$:} \emph{If a characteristic subgroup $K$ of $G$ is proper, then $K$ is a nilpotent $\tcn{r(t)}'$-group.}

\begin{proof}  \let\qed\relax
	The proper section $K$ of $G$ is an extension of a characteristic $\tcn{r(t)}$-group $L$ by a nilpotent $\tcn{r(t)}'$-group $K/L$. This $L$ is also characteristic in $G$. Claim $1$ implies that $L = 1$. So $K \cong K/L$ is a nilpotent $\tcn{r(t)}'$-group.
\end{proof}

\textbf{Claim $3$:} \emph{If a characteristic subgroup $A$ of $G$ is non-trivial, then the quotient $G/A$ is a $\tcn{r(t)}$-group or a nilpotent $\tcn{r(t)}'$-group.}

\begin{proof} \let\qed\relax
	The proper section $G/A$ of $G$ has a characteristic $\tcn{r(t)}$-subgroup $K/A$ with nilpotent $\tcn{r(t)}'$-quotient $(G/A)/(K/A) \cong G/K$. By lifting, we obtain a characteristic subgroup $K$ of $G$. If $K = G$, then $G/A = K/A$ is a $\tcn{r(t)}$-group. So we may suppose that $K$ is a proper subgroup. Claim $2$ then implies that the group $K$, and therefore the section $K/A$, is a $\tcn{r(t)}'$-group. So $K/A$ is both a $\tcn{r(t)}$-group and a $\tcn{r(t)}'$-group, and therefore the trivial group. We conclude that $G/A = G/K$ is a nilpotent $\tcn{r(t)}'$-group. 
\end{proof}

Let $F$ be the (first upper) Fitting subgroup of $G$. \\

\textbf{Claim $4$:} \emph{The quotient $G/F$ is non-trivial and elementary-abelian. So there is a prime $p$ and a natural number $l$ such that $G/F \cong (\Z_p^l,+)$.}

\begin{proof} \let\qed\relax
	As a counter-example to the theorem, $G$ is not nilpotent. So $F$ is a proper subgroup of $G$ and $G/F$ is non-trivial. Now every proper, characteristic subgroup $S/F$ of $G/F$ lifts to a proper, non-trivial, characteristic subgroup $S$ of $G$. By claim $2$, such an $S$ is nilpotent, and therefore contained in $F$. So the solvable group $G/F$ is characteristically simple, and therefore elementary-abelian.
\end{proof}

\textbf{Claim $5$:} \emph{The group $F$ is a $q$-group, for some prime $q $ not dividing $\tcn{r(t)}$.}

\begin{proof} \let\qed\relax
	Suppose that the nilpotent group $F$ is not a $q$-group. Such a group $F$ then decomposes as the direct product of proper, non-trivial, characteristic subgroups $A$ and $B$: $F = A \times B$. According to claim $2$, the non-trivial section $F/A \cong B$ of $G/A$ is a $\tcn{r(t)}'$-group. So $G/A$ is not a $\tcn{r(t)}$-group. According to claim $3$, the group $G/A$ is then nilpotent. Similarly, we obtain that $G/B$ is nilpotent. So $G/(A \cap B)$ is nilpotent (a contradiction). 
\end{proof}

\textbf{Claim $6$:} \emph{$F$ is non-trivial and elementary-abelian. So there is a natural number $k$ such that $F \cong (\Z_q^k,+)$.}

\begin{proof} \let\qed\relax
	Since $G$ is solvable and non-trivial, so is $F$. Let $\Phi(F)$ be the Frattini subgroup of $F$. We now consider the characteristic series $1 \leq \Phi(F) \leq F < G$ of $G$. We suppose that $F$ is not elementary abelian and we will derive a contradiction. In this case, the inclusions of the series are all strict. Claim $2$ implies that $F$ is a $\tcn{r(t)}'$-group, so that $G/\Phi(F)$ is not a $\tcn{r(t)}$-group. Claim $3$ now implies that $G/\Phi(F)$ is nilpotent, so that also $G$ is nilpotent (a contradiction).
\end{proof}

Since $G$ is not nilpotent, $p \neq q$. \\

\textbf{Claim $7$:} \emph{$F$ has an $\alpha$-invariant complement $P$ in $G$.}

\begin{proof} \let\qed\relax
	Zassenhaus' theorem implies that $G \cong F \rtimes (G/F)$. By using lemma \ref{LemmaRegularPrelims}, we can find a complement $P$ of $F$ in $G$ that is $\alpha$-invariant.
\end{proof}

We now consider the $k$-dimensional vector space $V := \F^k$ over the algebraically-closed field $\F:= \overline{\F}_q$. \\

\textbf{Claim $8$:} \emph{The vector space $V$ admits a non-trivial, elementary-abelian $p$-group of automorphisms $R \leq \operatorname{GL}(V)$ and an automorphism $f \in \operatorname{N}_{\operatorname{GL}(V)}(R) \setminus \operatorname{C}_{\operatorname{GL}(V)}(R)$ such that $r(f) = \{ 0_V\}$.}

\begin{proof} \let\qed\relax
	Since $P$ is $\alpha$-invariant, the group $P \rtimes \langle \alpha \rangle$ is well-defined and it acts naturally on $F$ via conjugation within the larger group $(F \rtimes P) \rtimes \langle \alpha \rangle = F \rtimes (P \rtimes \langle \alpha \rangle )$. Let this action be given by the map $\map{\theta}{P \rtimes \langle \alpha \rangle}{\operatorname{Aut}(F)}$. Since the Fitting subgroup is self-centralizing, this action is faithful on $P$, so that $\theta(P)$ is a non-trivial $p$-subgroup of $\operatorname{Aut}(F)$. It is clear that $\theta(\alpha)$ normalizes $\theta(P)$. Since $\alpha$ is fix-point-free on $P$, $\theta(\alpha)$ does not centralize $\theta(P)$. Trivially, $r(t)$ is an identity of the automorphism $\theta(\alpha)$ of $F$. \\
	
	We now identify the Fitting subgroup $F$ of $G$ with the additive group of the vector space $\F_q^k$. We then define $\overline{R}$ as the subgroup of $\operatorname{GL}(\F_q^k)$ corresponding with $\theta(P)$ and $\overline{f}$ as the automorphism of $\operatorname{GL}(\F_q^k)$ corresponding with $\theta(\alpha)$. By extending the scalars from $\F_q$ to $\overline{\F}_q$, we naturally obtain the desired subgroup $R$ of $\operatorname{GL}(V)$ and the automorphism $f \in \operatorname{GL}(V)$.
\end{proof}

We had already observed that $p \neq q$. This implies that the operators of the abelian group $R$ can be diagonalized simultaneously. So let $\chi_1 , \ldots , \chi_l$ be the distinct characters of $R$ with coefficients in $\mathbb{F}^\times$ and let $V_{\chi_i} := \{ v \in V | \forall b \in R : b(v) = \chi_i(b) \cdot v \}$ be the (by definition non-trivial) character-space corresponding with such a character $\map{\chi_i}{R}{\mathbb{F}^\times}$. Then we have the decomposition $V = V_{\chi_1} \oplus \cdots \oplus V_{\chi_l}$ of $V$ into its $R$-character-spaces. Since $f$ normalises $R$, we can define, for each $n \in \Z$ and each character $\chi_i$, a new character $(f^n \ast \chi_i)$ by the rule: $\forall b \in R : (f \ast \chi_i)(b) := \chi_i(f^{-n} \circ b \circ f^n).$ So the group $\langle f \rangle$ acts as permutations on the set of character spaces according to the rule $$f^n(V_{\chi_i}) = V_{(f^n \ast \chi_i)}.$$

\textbf{Claim $9$:} \emph{The automorphism $f$ maps some character space, say $V_{\chi_m}$, onto a different character space.} 

\begin{proof} \let\qed\relax
	Since $f$ does not centralise $R$, there exists an element $b \in R$ such that $f^{-1 } \circ b \circ f \neq b$. These operators $f^{-1} \circ b \circ f$ and $b$ must differ in some element of some character-space, say $v \in V_{\chi_m}$. 
	Now, if $f(V_{\chi_m}) = V_{\chi_m}$, then we obtain the contradiction $(f^{-1} \circ b \circ f)(v) = (f^{-1} \circ b )(f(v)) = f^{-1}(\chi_m(b) \cdot f(v)) = \chi_m(b) \cdot f^{-1}(f(v)) = b(v).$ So we may indeed conclude that $f(V_{\chi_m}) \neq V_{\chi_m}$.	
\end{proof}

Let $u$ be the minimal natural number such that $f^u(V_{\chi_m}) = V_{\chi_m}$. By the previous remark, we necessarily have $u \geq 2$. \\

\textbf{Claim $10$:} \emph{We have $V_{\chi_m} = \{0_V\}$.}

\begin{proof} \let\qed\relax
	Let $v$ be any vector of $V_{\chi_m}$ and let us show that $v = 0_V$. We first consider the $u$-periodic decomposition $r(t) = \sum_j r_{u,j}(t)$ of $r(t)$ into partial sums. By evaluating this decomposition in $f$, we obtain the equality of $V$-endomorphisms: $r(f) = r_{u,0}(f) + r_{u,1}(f) + \cdots + r_{u,u-1}(f).$ By further evaluating this decomposition in the distinguished vector $v \in V_{\chi_m}$, we obtain the decomposition: $ (r(f))(v) = (r_{u,0}(f))(v) + \cdots + (r_{u,u-1}(f))(v) .$ By assumption, the map $r(f)$ vanishes on all of $V$, so that the left-hand side of this decomposition is the trivial vector, $0_V$. By construction, we also have the inclusions of vector spaces $r_{u,0}(f)(V_{\chi_m}) \subseteq V_{\chi_m},$ $ \ldots,$ $r_{u,u-1}(f)(V_{\chi_m}) \subseteq f^{u-1}(V_{\chi_m})$. Since the character-spaces $V_{\chi_m}$, $f(V_{\chi_m}),$ $\ldots , f^{u-1}(V_{\chi_m})$ are are distinct, they are linearly independent. So we may conclude that all of the terms in this decomposition vanish: \begin{equation}
	(r_{u,0}(f))(v) = \cdots  = (r_{u,u-1}(f))(v) = \{ 0_V \} . \label{EqPartSumVanish}
	\end{equation}	
	We had already observed that $q $ does not divide $\tcn{r(t)}$. By Bezout's theorem, there exist polynomials $s_{0}(t) , \ldots  , s_{u-1}(t) \in \Z[t]$ such that $s_0(t) \cdot r_{u,0}(t) + \cdots + s_{u-1}(t) \cdot r_{u,u-1}(t) \equiv 1 \operatorname{mod} q.$ By evaluating this decomposition in $f$, we obtain the decomposition of $V$-endomorphisms \begin{equation}
	s_0(f) \circ r_{u,0}(f) + \cdots + s_{u-1}(f) \circ r_{u,u-1}(f) = \id_V. \label{EqPartSumEqual}
	\end{equation} By further evaluating this expression in the distinguished vector $v \in V_{\chi_m}$, we obtain 
	\begin{eqnarray*}
		v &=& \id_V(v) \\
		&=^{\eqref{EqPartSumEqual}}& (s_0(f) \circ r_{u,0}(f))(v)  + \cdots +  (s_{u-1}(f) \circ r_{u,u-1}(f)) (v) \\
		&=& s_0(f) (r_{u,0}(f)(v))  + \cdots + s_{u-1}(f) (r_{u,u-1}(f) (v)) \\
		&=^{\eqref{EqPartSumVanish}}& s_0(f) (0)  + \cdots + s_{u-1}(f) (0) \\
		&=& 0_V.\qedhere
	\end{eqnarray*} 
\end{proof}

But the character spaces are non-trivial \emph{by definition}. This contradiction finishes the proof of theorem \ref{CorollaryDecompositionKN}.

\begin{corollary} \label{CorNoBThenNil}
	Consider a finite, solvable group $G$, together with a fix-point-free automorphism $\map{\alpha}{G}{G}$, and an identity $r(t)$ of the automorphism. If $G$ has no $\tcn{r(t)}$-torsion, then $G$ is nilpotent.
\end{corollary}

\subsubsection{General case}

\begin{proof}
	We suppose that the conclusion is false and we will deduce a contradiction (this is essentially Thompson's proof of the Frobenius conjecture \cite{ThompsonFixPointFreeAutomorphisms}, and we include it here for completeness). Let $(G,\alpha,r(t))$ be a counter-example of minimal order $|G|$. We note that every proper, $\alpha$-invariant section of $G$ is then nilpotent. If $|G|$ is the power of a single prime, then $G$ is nilpotent (a contradiction). So some odd prime $p$ divides $|G|$. Lemma \ref{LemmaRegularPrelims} gives us a $p$-Sylow subgroup satisfying $\alpha(P) = P$. We now distinguish between two cases. \\
	
	Suppose first that $P$ has a non-trivial, normal, $\alpha$-invariant subgroup $H$ with $\operatorname{N}_G(H) = G$. Then the section $G/H$ is proper and $\alpha$-invariant. So $G$ is the extension of the nilpotent group $H$ by the nilpotent group $G/H$. The group $G$ is therefore solvable. Corollary \ref{CorNoBThenNil} now implies that $G$ is nilpotent (a contradiction). Suppose next that every non-trivial, normal, $\alpha$-invariant subgroup $H$ of $P$ has a normalizer $\operatorname{N}_G(P)$ that is properly contained in $G$. As a proper, $\alpha$-invariant section of $G$, this $\operatorname{N}_G(H)$ is nilpotent. So theorem $1$ of Thompson \cite{ThompsonNormalComplements} gives us a normal complement $K$ to $P$ in $G$: $G \cong K \rtimes P$. As a proper, $\alpha$-invariant section of $G$, this $K$ is nilpotent. As an extension of a nilpotent group $K$ by a nilpotent group $P$, our group $G$ is solvable. Corollary \ref{CorNoBThenNil} now implies that $G$ is nilpotent (a contradiction).
\end{proof}

\subsection{Proof of theorem \ref{MainTheoremB}}

\begin{proof}
	Let $G$ be a nilpotent group having an endomorphism $\map{\gamma}{G}{G}$ with a good identity $r(t)$, say of degree $d$. Suppose that $G$ has no $(\discr{r(t)} \cdot \prodant{r(t)})$-torsion. We then consider the Lie ring $\operatorname{L}(G) := \bigoplus_{n \in \N} \Gamma_n^\ast(G) / \Gamma_{n+1}^\ast(G)$ that corresponds with the isolators of the lower central series $(\Gamma_n(G))_{n \in \N}$ of $G$. This Lie ring has a number of good properties. It is also nilpotent and its class coincides with that of $G$. Moreover, $(\operatorname{L}(G),+)$ has no $(\discr{r(t)} \cdot \prodant{r(t)})$-torsion. And the Lie ring naturally admits an automorphism $\map{\overline{\alpha}}{\operatorname{L}(G)}{\operatorname{L}(G)}$ of Lie rings that satisfies $r(\overline{\alpha}) = 0_{\operatorname{L}(G)}$. So we need only apply corollary \ref{IntroTheoremLie} to conclude that $\operatorname{c}(G) \leq d^{2^d}$. 
\end{proof}

\subsection{Proof of corollary \ref{MainCorollaryWithClassification}} \label{SectionCorolls} 

\begin{proof}
	Let $\map{\alpha}{G}{G}$ be the fix-point-free automorphism of the finite group $G$. Rowley's theorem \ref{TheoremA} implies that $G$ is solvable. Corollary \ref{CorNoBThenNil} implies that $G$ is of the form $B \rtimes N$, where $B$ is a characteristic $\tcn{r(t)}$-subgroup and $N$ is a nilpotent, $\tcn{r(t)}'$-subgroup. According to lemma \ref{LemmaRegularPrelims}, we may further assume that $N$ is $\alpha$-invariant. Corollary \ref{MainCorollaryNoClassification} implies that $N$ is of the form $C \times D$, where $C$ is a nilpotent $(\discr{r(t)} \cdot \prodant{r(t)})$-group and $D$ is a nilpotent group of class at most $d^{2^d}$.
\end{proof}

\subsection{Proof of corollary \ref{CorollIrreducibleConstants}} \label{MainCorollIrred}


\begin{definition}[$\operatorname{a},\operatorname{b},\operatorname{c}$] Let $r(t) \in \Z[t]$ be irreducible over $\Q$. Suppose first that $r(0) = 0$ or $r(1) = 0$. Then we define $\operatorname{a} := \operatorname{b} := \operatorname{c} := 1$. Suppose next that $r(0) \cdot r(1) \neq 0$. Then we define $\operatorname{a} := r(1)$. Let $u$ be the largest natural number such that $r(t) \in \Z[t^u]$. Let $s(t)$ be the unique $s(t) \in \Z[t]$ such that $r(t) = s(t^u)$. Then we define $\operatorname{b} := \tcn{s(t)}$ and $\operatorname{c} := \discr{s(t)} \cdot \prodant{s(t)}$.
\end{definition}

\begin{proof}(Corollary \ref{CorollIrreducibleConstants})
		We can now prove the theorem case-by-case.  Suppose first that $r(0) = 0$. Then $r(t) = 0$ and we have, for all $x \in G$, the equality $\alpha(x) = 1_G$. So $G$ is trivial.
		Suppose next that $r(1) = 0$. Then $r(t) = t-1$ and we have, for all $x \in G$, that $\alpha(x) = x$. Since $\alpha$ is fix-point-free, we conclude again that $G$ is the trivial group.
		Finally, suppose that $r(0) \cdot r(1) \neq 0$. Then $\operatorname{a}$ is clearly non-zero. Since $r(t)$ is irreducible, so is $s(t)$. Since $s(0) \cdot s(1) \cdot s(t)$ has no divisors of the form $h(t^{v+1})$ for any $v \in \N$, the polynomial is good. So proposition \label{PropositionCongNonZero} implies that $\operatorname{b} \neq 0$. Finally, since $s(t) \neq 0$, we see that $\operatorname{c} \neq 0$. Let $A$ be some $\operatorname{a}$-Hall subgroup $A$ of $G$. Lemma \ref{LemmaRegularPrelims} gives us some $a'$-Hall subgroup $H$ of $G$ that is $\alpha$-invariant. Then $G$ factors as $A \cdot H$. Now we consider the restriction of $\alpha$ to $H$, and apply corollary \ref{MainCorollaryWithClassification} to obtain the desired subgroups $B$, $C$, and $D$ of $H$.
\end{proof}

\section{Applications} \label{SectionExamples}

\subsection{Linear identities}

\begin{definition}[$n$-abelian groups and power endomorphisms]
	Let $n$ be a natural number. A group $G$ is said to be \emph{$n$-abelian} (or \emph{$n$-commutative}) if the $n$'th power map $\mapl{\gamma_n}{G}{G}{x}{x^n}$ is an endomorphism. An endomorphism $\map{\gamma}{G}{G}$ of the group $G$ is said to be a \emph{(universal) power endomorphism} if there is a natural number $n$ such that, for all $x \in G$, we have $\gamma(x) = x^n$.
\end{definition} 

These $n$-abelian groups have been studied by various authors, including Levi \cite{Levi}, Baer \cite{Baer}, Schenkman---Wade \cite{SchenkmanWade}, and J. Alperin \cite{AlperinPowerAutomorphism}. They were ultimately classified by Alperin.

\begin{theorem}[Alperin \cite{AlperinPowerAutomorphism}]
	Let $n > 1$ be a natural number. The finite $n$-abelian groups are the subdirect product of a finite abelian group, a finite group of exponent dividing $n$, and a finite group of exponent dividing $(n-1)$.
\end{theorem}

Now let $G$ be an $n$-abelian group, for some $n > 1$. Then $n^2-n \neq 0$ and $r_{n}(t) := -n + t$ is clearly a monic and monotone identity of the $n$'th power endomorphism $\gamma_n$, i.e.: $r_n(\gamma_n) = 1_G$. Moreover, this $\gamma_n$ is a fix-point-free automorphism if $G$ has no $(n^2 - n)$-torsion. More generally, we may consider endomorphisms with a linear identity \emph{that is not necessarily monic or monotone}. 

\begin{corollary}
	Let $G$ be a finite group with an endomorphism with a linear identity, say $-n + m \cdot t \in \Z[t]$. Suppose that $(m \cdot n \cdot (m-n)) \neq 0$.  Then either $G$ has $(m \cdot n \cdot (m-n))$-torsion, or $G$ is abelian.
\end{corollary}

\begin{proof}
	Set $r(t) := -n + m \cdot t$. It is not difficult to verify that $r(0) \cdot r(1) \cdot \tcn{r(t)} \cdot \discr{r(t)} \cdot \prodant{r(t)}$ divides a natural power of $m \cdot n \cdot (m - n)$. So the endomorphism is a fix-point-free automorphism of the group and we may apply corollary \ref{MainCorollaryNoClassification}.
\end{proof}

\subsection{Cyclotomic identities}

\begin{definition}[Splitting and cyclotomic automorphisms]
	Let $n > 1$ be a natural number. An automorphism $\map{\alpha}{G}{G}$ is said to be \emph{$n$-splitting} if the polynomial $\Psi_n(t) := 1 + t + \cdots + t^{n-1}$ is a monotone identity of $\alpha$, i.e.: $\Psi_n(\alpha) = 1_G$. We say that the automorphism is \emph{$n$-cyclotomic} if the $n$'th cyclotomic polynomial $\Phi_n(t)$ is a monotone identity of $\alpha$, i.e.: $\Phi_p(\alpha) = 1_G$. More generally, we say that $\alpha$ is \emph{splitting} (resp. \emph{cyclotomic}) if it is $m$-splitting (resp. $m$-cyclotomic) for some natural $m > 1$.
\end{definition}

It is well-known that $\Psi_n(t) = \Phi_n(t)$ if $n$ is a prime. So the automorphism $\alpha$ in theorem \ref{TheoremA} is an $|\alpha|$-splitting automorphism, while the automorphism $\alpha$ in theorems \ref{TheoremB} and \ref{TheoremC} is an $|\alpha|$-cyclotomic automorphism (cf. remark \ref{RmkFiniteOrderToCyclo}). We now recall three companions of theorems \ref{TheoremA}, \ref{TheoremB}, and \ref{TheoremC}.


\begin{theorem}[Ersoy \cite{ErsoySplittingAutomorphisms}] \label{TheoremD} Let $n \in \N$ be odd. If a finite group $G$ has an $n$-splitting automorphism, then $G$ is solvable. \end{theorem}

Theorem \ref{TheoremD} is a (partial) generalization of theorem \ref{TheoremA} and it based on the classification of the finite simple groups. Such automorphisms naturally appear in the study of automorphisms with finite order and with finite Reidemeister-number \cite{JabaraReidemeister,BettioJabaraWehrfritz}. 

\begin{theorem}[Hughes---Thompson \cite{HughesThompson}; Kegel \cite{KegelSplit}] \label{TheoremE} Let $p$ be a prime number. If a finite group $G$ has a $p$-cyclotomic automorphism, then $G$ is nilpotent. \end{theorem}

Theorem \ref{TheoremE} clearly generalizes theorem \ref{TheoremB}. The solvability of $G$ was proven by Hughes---Thompson using the fundamental results of Hall and Higman about minimal polynomials of operators on finite-dimensional vector spaces \cite{HallHigmanReduction}. The nilpotency of $G$ is due to Kegel. These $p$-split automorphisms naturally appear in the study of almost-fix-point-free automorphisms of prime order by Bettio, Endimioni, Jabara, Wehrfritz, Zappa, and others  \cite{EndimioniPolycyclicRegularAutoPrime,BettioJabaraWehrfritz}.

\begin{theorem}[E. Khukhro \cite{KhukhroSplitPrimeFiniteGroup}] For each natural number $d$ and each prime $p$, there exists a natural number $\operatorname{C}(d,p)$ with the following property. \label{TheoremF} If a finite $p$-group $G$ on $d$ generators has a $p$-cyclotomic automorphism, then $\operatorname{c}(G) \leq \operatorname{C}(d,p)$. \end{theorem}

In combination with theorem \ref{TheoremC}, this gives us a bound on the class of all finite, nilpotent groups with a $p$-split automorphism. One could hope to obtain an upper bound on $\operatorname{c}(G)$ that depends only on $p$ (as in theorem \ref{TheoremC}), but examples show that this is not possible. The theorem applies, in particular, to finite $p$-groups with a partition. Moreover, Kostrikin's positive solution \cite{KostrikinBurnsideProblem} of the restricted Burnside problem in exponent $p$ corresponds with the automorphism $x \mapsto x$. \\

Examples show that theorem \ref{TheoremD} cannot be extended to all natural numbers $n$. But we can generalize theorem \ref{TheoremE} of Hughes---Thompson and Kegel, theorem \ref{TheoremC} of Higman and Kreknin---Kostrikin, as well as theorem \ref{TheoremF} of Khukhro:.

\begin{theorem} \label{TheoremTHHUTHKEG}
	Compact groups with a cyclotomic automorphism are locally-nilpotent.
\end{theorem}

\begin{theorem} \label{TheoremHKKK}
	Consider a locally-nilpotent group $G$ with an $n$-cyclotomic automorphism. Let $p$ be the largest prime divisor of $n$. Then every $k$-generated subgroup $N$ of $G$ is nilpotent of class $\operatorname{c}(N) \leq \max \lbrace (p-1)^{2^{(p-1)}} , \operatorname{C}((2p-3) \cdot k , p) \rbrace.$
\end{theorem}

Here, we always assume that compact groups are Hausdorff. These theorems allow us to to recover (and extend) two theorems of Jabara about automorphisms with a finite Reidemeister-number.

\begin{corollary}[Jabara \cite{JabaraReidemeister}]
	\label{CorJabaraIntro} Consider a residually-finite group $G$ admitting an automorphism $\map{\alpha}{G}{G}$ of prime order $p$. If the Reidemeister-number of $\alpha$ is finite, then $G$ has an $\alpha$-invariant subgroup $N$ of finite index that is nilpotent of class $\operatorname{c}(N) \leq {(p-1)}^{2^{(p-1)}}. $ \end{corollary}

\begin{corollary}[Jabara \cite{JabaraReidemeister}]
	\label{CorJabaraBIntro}
	Consider a finitely-generated, solvable group $G$ with an automorphism $\map{\alpha}{G}{G}$ of prime order $p$. If the Reidemeister-number $n$ of $\alpha$ is finite, then $G$ has a finite-index subgroup $N$ that is nilpotent of class $\operatorname{c}(N) \leq (p-1)^{2^{(p-1)}}$ and $\operatorname{dl}(G) \leq 2^{2^n} + \operatorname{A}(p,n) + (p-1)^{2^{(p-1)}}.$
\end{corollary}

Here, $\map{\operatorname{A}}{\mathbb{P} \times \N}{\N}$ is the map in J. Alperin's theorem $1$ of \cite{AlperinAutomorphisms}. We note that Jabara's proofs of these corollaries implicitly used the classification of the finite simple groups, as well as Zel$'$manov's positive solution of the restricted Burnside problem in arbitrary exponent, and the theorems of Hartley, Hartley---Meixner, Fong, and Khukhro (cf. corollary $5.4.1$ of \cite{KhukhroNilpotentGroupsAutomorphisms}). Our proofs will be closely modeled on those of Jabara, but they will avoid all of these rather difficult results.

\subsubsection{Proof of theorems \ref{TheoremTHHUTHKEG} and \ref{TheoremHKKK}}

For a natural number $n$, we let $\overline{n}$ be the radical of $n$: the product of the distinct primes that divide $n$.

\begin{proposition} \label{PropFiniteCycloNilpotent} Let $n > 1$ be a natural number. Let $p$ be the largest prime dividing $n$. Consider a finite group $G$ on $k$ generators with an $n$-cyclotomic automorphism $\map{\alpha}{G}{G}$. $(i.)$ Then $G$ is nilpotent.
	$(ii.)$ If $\overline{n} = p$ and $G$ has no $p$-torsion, then $\operatorname{c}(G) \leq {(p-1)}^{2^{(p-1)}}.$
	$(iii.)$  If $\overline{n} = p$ and if $G$ is a $p$-group, then $\operatorname{c}(G) \leq \operatorname{C}(k,p).$
	$(iv.)$  If $\overline{n} \neq p$, then $\operatorname{c}(G) \leq {(p-1)}^{2^{(p-1)}}.$ So $\operatorname{c}(G) \leq \max \{ {(p-1)}^{2^{(p-1)}} , \operatorname{C}(k,p) \}.$
\end{proposition}

\begin{proof}
	The automorphism $\alpha^{n/\overline{n}}$ is $\overline{n}$-cyclotomic. So, after replacing $\alpha$ with $\alpha^{n/\overline{n}}$, we may assume that $n = \overline{n}$. $(i.)$ If $n$ is a prime, then we apply theorem  \ref{TheoremE}. If $n$ is not a prime, then $\Phi_n(1) \cdot \tcn{\Phi_n(t)} = 1$ (cf. lemma \ref{LemFormulaCyclo} and proposition \ref{PropCongCycloComputed}), so that we may apply corollary \ref{MainCorollaryNoClassification}. In either case, $G$ is nilpotent. $(ii.)$ We may apply theorem \ref{TheoremC} or corollary \ref{MainCorollaryWithClassification}. $(iii.)$ We may apply theorem \ref{TheoremF}. $(iv.)$ We need only show that an arbitrary $q$-Sylow subgroup $Q$ of $G$ satisfies $c(Q) \leq (p-1)^{2^{(p-1)}}$. We let $L$ be the Lie ring of $Q$ corresponding with the lower central series of $Q$.  Then the natural Lie automorphism $\map{\overline{\alpha}}{L}{L}$ satisfies $\Phi_n(\overline{\alpha}) = 0_L$. Since $n$ is composite, there exists a prime $l$, distinct from $q$, that divides $n$. We may then apply \eqref{FormulaRadical} of lemma \ref{LemFormulaCyclo} in order to obtain a natural number $u$ such that $\Phi_n(t)$ divides $\Phi_l(t^u)$ in the ring $\Z[t]$. Then the Lie automorphism $\overline{\alpha}^u$ satisfies $\Phi_l(\overline{\alpha}^u) = 0_L$. The second claim now gives us $\operatorname{c}(Q) = \operatorname{c}(L) \leq (l-1)^{2^{(l-1)}} \leq (p-1)^{2^{(p-1)}}.$ \end{proof}

This already proves corollaries \ref{TheoremTHHUTHKEG} and \ref{TheoremHKKK} in the finite case. In order to extend these results from finite groups to locally-(residually-finite) groups, we make some elementary observations. 

\begin{lemma} \label{LemLocal} Consider a group $G$ with an automorphism $\map{\alpha}{G}{G}$ and a monotone identity $r(t) := a_0 + a_1 \cdot t + \cdots + a_n \cdot t^n \in \Z[t]$ of degree $n \geq 1$. Suppose that $a_0,a_n \in \{ 1,-1\}$. Let $H$ be a finitely-generated subgroup of $G$ and define the subgroup $\widetilde{H} := \langle \alpha^{-n+1}(H),\ldots,\alpha^{-1}(H),H,\alpha(H),\ldots,\alpha^{n-1}(H) \rangle$ of $G$. 
	Then $\widetilde{H}$ is $\langle \alpha \rangle$-invariant and $\operatorname{d}(\widetilde{H}) \leq (2 n -1) \cdot \operatorname{d}(H).$ Suppose, moreover, that $G$ is locally-(residually-finite). Then $\widetilde{H}$ admits a family $(\widetilde{H}_i)_i$ of $\widetilde{H}$-characteristic, finite-index subgroups with trivial intersection, so that 	$r(t)$ is a monotone identity of all the induced automorphisms $\map{\alpha_{\widetilde{H}/\widetilde{H}_i}}{\widetilde{H}/\widetilde{H}_i}{\widetilde{H}/\widetilde{H}_i}$. 
\end{lemma}

\begin{proof} Let us show that $\widetilde{H}$ is invariant under $\alpha$ and $\alpha^{-1}$. Since $a_n \in \{ 1,-1\}$, we observe that $\alpha(\widetilde{H}) \subseteq \langle \alpha^{-n+2}(H) , \ldots, \alpha^{n-1}(H) , \alpha^n(H) \rangle \subseteq \langle \widetilde{H} , \alpha^{n-1}(H) ,\ldots, \alpha(H) , H \rangle \subseteq \widetilde{H}.$ Since $a_0 \in \{1,-1\}$, we similarly observe that $\alpha^{-1}(\widetilde{H}) \subseteq \widetilde{H}$. So $\widetilde{H}$ is $\langle \alpha \rangle$-invariant. By construction, we also have $\operatorname{d}(\widetilde{H}) \leq (2n-1) \cdot \operatorname{d}(H)$. Finally, we suppose that $G$ is locally-(residually-finite), so that the subgroup $\widetilde{H}$ is residually-finite. Let $\widetilde{H} = N_1 \trianglerighteq N_2 \trianglerighteq \cdots$ be a normal series of finite-index subgroups with finite intersection. Since $\widetilde{H}$ is finitely-generated, every subgroup $N_i$ contains the $\widetilde{H}$-characteristic subgroup $\widetilde{H}_i := \bigcap_{\beta \in \operatorname{Aut}(\widetilde{H})} \beta(N_i)$ of finite index in $\widetilde{H}$ (the characteristic core). Since $\bigcap_i \widetilde{H}_i \subseteq \bigcap_i N_i = \{1_{\widetilde{H}}\}$, we see that $(\widetilde{H}_i)_i$ is the desired series. 	\end{proof}

\begin{proposition} \label{ThApplCycloGeneral}
	Consider a locally-(residually-finite) group $G$ with an $n$-cyclotomic automorphism. Then every $k$-generated subgroup $H$ of $G$ is nilpotent and $\operatorname{c}(H) \leq \max \{ (p-1)^{^{(p-1)}} , \operatorname{C}( (2p - 3) \cdot k,p)\},$ where $p$ is the largest prime dividing $n$.
\end{proposition}

\begin{proof}
	Let $\map{\alpha}{G}{G}$ be the $n$-cyclotomic automorphism of the group. Let $\widetilde{H}$ be the $\langle \alpha \rangle$-invariant subgroup of $G$ on $\widetilde{k} := (2 \varphi(n)-1)$ generators containing $H$ from lemma \ref{LemLocal}. According to lemma \ref{LemLocal}, $\widetilde{H}$ is residually-(a finite group on $\widetilde{k}$ generators with an $n$-cyclotomic automorphism). So proposition \ref{PropFiniteCycloNilpotent} implies that $H$ is nilpotent of class $\operatorname{c}(H) \leq \operatorname{c}(\widetilde{H}) \leq \max \{ {(p-1)}^{2^{(p-1)}} , \operatorname{C}(\widetilde{k}(2\varphi(n)-1),p) \}$.
\end{proof}

\begin{proof}(Theorem \ref{TheoremTHHUTHKEG})
	The theorems of Peter---Weyl and Mal$'$cev imply that $G$ is locally-(residually-finite). So we need only apply proposition \ref{ThApplCycloGeneral}.
\end{proof}

\begin{proof}(Theorem \ref{TheoremHKKK})
	By Mal$'$cev's theorems, the group is locally-linear and therefore locally-(residually-finite). So we need only apply proposition \ref{ThApplCycloGeneral}.
\end{proof}

\subsubsection{Proof of corollaries \ref{CorJabaraIntro} and \ref{CorJabaraBIntro}}

\begin{proof}(Corollary \ref{CorJabaraIntro}) It is easy to show that $G$ has an $\langle \alpha \rangle$-invariant, finite-index subgroup $M$ such that the induced automorphism $\map{\alpha_M}{M}{M}$ is $p$-cyclotomic (cf. lemma $5$ of \cite{JabaraReidemeister}). It is also easy to see that $\alpha_M$ fixes only finitely-many elements (cf. lemma $1$ and $4$ of \cite{JabaraReidemeister}). As a subgroup of $G$, this $M$ is residually-finite and therefore locally-nilpotent by proposition \ref{ThApplCycloGeneral}. So we may consider the $p$-Sylow subgroup $P$ of the torsion subgroup $T$ of $G$. This $P$ is a characteristic subgroup of $G$ such that the locally-nilpotent factor $G/P$ has no $p$-torsion. \newline

Let us first show that this $P$ is finite. Since $G$ is residually-finite, so is its subgroup $P$. Since $\alpha$ fixes only finitely-many elements of $G$, we may select a finite-index subgroup $P_0$ of $P$ such that $P_0 \cap C_G(\alpha) = \{1_G\}$. After replacing $P_0$ with the finite-index subgroup $\bigcap_{0 \leq i \leq p-1} \alpha^i(P_0)$ of $P$, we may further assume that $P_0$ is $\alpha$-invariant. So the restriction of $\alpha$ to the locally-nilpotent $p$-group $P_0$ is a fix-point-free, $p$-cyclotomic automorphism. But it is well-known that such a group $P_0$ is trivial (cf. \cite{HigmanGroupsAndRings} and \cite{KhukhroNilpotentGroupsAutomorphisms}). We conclude that $P$ is indeed a finite group. Since $P$ is a finite $p$-group, $P$ is nilpotent. Theorem \ref{MainTheoremB} implies that also the quotient $G/P$ is nilpotent. We see, in particular, that $G$ is solvable, and therefore nilpotent (by Khukhro's theorem of \cite{KhukhroSplitSolvableNilpotent}).
\end{proof}

\begin{proof}(Corollary \ref{CorJabaraBIntro})
	We re-consider Jabara's proof of theorem $B$ in \cite{JabaraReidemeister} and we replace theorem $A$ of \cite{JabaraReidemeister} with corollary \ref{CorJabaraIntro}. This way, we obtain the subgroup $N$ without relying on the classification. By assumption, $G$ is finitely-generated. So, after replacing $N$ with $\bigcap_{\beta \in \operatorname{Aut}(G)} \beta(N)$, we may further assume that $N$ is characteristic in $G$. Then the induced automorphism $\map{\alpha_{G/N}}{G/N}{G/N}$ on the finite, solvable group $G/N$ fixes at most $2^{2^n}$ elements (cf. lemma $1$ and $4$ of \cite{JabaraReidemeister}) and its order divides $p$. If $\alpha_{G/N}$ has order $1$, then $\vert G/N \vert \leq 2^{2^n}$. Else, we may apply Alperin's theorem $1$ of \cite{AlperinAutomorphisms} to $G/N$ and $\alpha_{G/N}$ in order to conclude that $\operatorname{dl}(G/N) \leq \operatorname{A}(p,m)$. Since $\operatorname{dl}(G) \leq \operatorname{dl}(G/N) + \operatorname{dl}(N)$, we are done.
\end{proof}

\subsection{Anosov identities}

\begin{definition}[Anosov polynomial]
	A monic polynomial $r(t) \in \Z[t]$ is said to be \emph{Anosov} if and only if $r(0) \in \{-1,1\}$ and $r(t)$ has no roots of modulus one.
\end{definition}

We recall that Anosov polynomials naturally appear in the study of Anosov-diffeomorphisms on compact manifolds (e.g. \cite{DekimpeAnosovOnFreeNilpotent,PayneAnosov,DereNewMethodsAnosov}) and we recall one such situation in particular:

\begin{theorem}[Manning -- \cite{ManningAnosov}] A nil-manifold $M$ admits an Anosov diffeomorphism if and only if the fundamental group $\pi_1(M)$ of $M$ admits an automorphism $\map{\alpha}{\pi_1(M)}{\pi_1(M)}$ with an Anosov identity. \end{theorem}

In fact: every \emph{known} example of an Anosov-diffeomorphism on a compact manifold is topologically-conjugated to an infra-nil-manifold endomorphism (cf. \cite{DekimpeWhatReallyShouldBe}) and it is conjectured that there are no other examples (cf. Smale's problem $3.5$ in \cite{SmaleAnosov}). So it makes sense to ask which (fundamental) groups admit an automorphism with an Anosov identity. 
We have a partial rersult:

\begin{proposition} 
	Consider a finitely-generated group $G$, together with an automorphism admitting an Anosov identity $r(t)$, say of degree $d$. If $G$ has a $p$-congruence system such that $\discr{r(t)} \cdot \prodant{r(t)} \not \equiv 0 \operatorname{mod} p$, then $G$ has a nilpotent subgroup $N$ of finite index and of class $\operatorname{c}(N) \leq d^{2^d}$.
\end{proposition}

\begin{proof} Here we do not assume that the $p$-congruence system comes with a bound: we are using the terminology of definition $B.1$ in \cite{DDSMS}, rather than the original terminology of Lubotzky in \cite{LubotzkyLinear}.  After replacing $G$ with an appropriate characteristic subgroup of finite index, we may assume that $G$ is finitely-generated and residually-(a finite $p$-group). Let us consider the Zassenhaus series $(\Delta_i(G,p))_{i \in \N}$ of $G$ with respect to the prime $p$. Since $\bigcap_i \Delta_i(G,p) = \{1_G \}$, we need only show that each  factor $G / \Delta_i(G,p)$ is nilpotent of class at most $d^{2^d}$. Let $\map{\alpha}{G}{G}$ be the automorphism. Since the series is characteristic, we obtain the induced automorphism $\map{\alpha_{G / \Delta_i(G,p)}}{G / \Delta_i(G,p)}{G / \Delta_i(G,p)}$ with identity $r(t)$. Since $G / \Delta_i(G,p)$ is a finite $p$-group and since the roots of an Anosov-polynomial form an arithmetically-free subset of $(\overline{\Q}^\times,\cdot)$, we may apply theorem \ref{MainTheoremB} in order to conclude that $\operatorname{c}(G / \Delta_i(G,p)) \leq d^{2^d}$. This finishes the proof. \end{proof}

\section{Closing remarks} \label{SubsectionAnosov}

Our main theorem \ref{MainCorollaryWithClassification} roughly states that finite groups with a fix-point-free automorphism satisfying a good identity are nilpotent of bounded class modulo some bad torsion. But the theorem, as it is stated, cannot be extended to \emph{all} polynomials. 

\begin{example}[Gross \cite{GrossFittingEx}]
	Let $n$ be a natural number and let $k$ be the number of prime divisors of $n$, counted with multiplicity. Let $p,q$ be distinct primes not dividing $n$. Then there are finite $(p \cdot q)$-groups of Fitting-height $k$ with a fix-point-free automorphism of order $n$.
\end{example}

So we propose the following (alternative) interpretation of Meta-Problem \ref{MetaProblem}.

\begin{problem*} For every $r(t) := a_0 + a_1 \cdot t + \cdots + a_d \cdot t^d \in \Z[t] \setminus \{0\}$, there exists some $h \in \N$ with the following property. Consider a finite group $G$, together with a fix-point-free automorphism $\map{\alpha}{G}{G}$, and suppose that, for every $x \in G$, we have the equality $x^{a_0} \cdot \alpha(x^{a_1}) \cdot \alpha^2(x^{a_2})\cdots \alpha^d(x^{a_d}) = 1_G.$ Then $G$ has Fitting-height $\operatorname{h}(G) \leq h$. \end{problem*}

Some evidence in favour of a positive solution exists. If $r(t)$ is the \emph{constant} polynomial $n$ with prime factorization $ \prod_i p_i^{e_i}$, then $\operatorname{h}(G) \leq \prod_i (2e_i +1)$ (cf. Shalev's lemma  \cite{ShalevCentralizersResiduallyFiniteTorsion}).
 If $r(t)$ is the \emph{linear} polynomial $-n+t$, then $\operatorname{h}(G) \leq 1$ (cf. Alperin's theorem \cite{AlperinPowerAutomorphism}).
 If $r(t)$ is a \emph{splitting} polynomial $\Psi_n(t) = (t^n-1)/(t-1)$, then $\operatorname{h}(G) \leq 7 \cdot n^2$. (cf.  Jabara's theorem \cite{JabaraFitting}).
 If $r(t)$ is a \emph{cyclotomic} polynomial $\Phi_n(t)$, then $\operatorname{h}(G) \leq 1$ (cf. theorem \ref{TheoremTHHUTHKEG}).

\paragraph{Acknowledgements.} The author would like to thank Efim Zel$'$manov and Lance Small for their hosting and their support during the first phase of the project at the University of California, San Diego. The author also thanks Jonas Der\'e (for helpful discussions about the construction of automorphisms in subsection \ref{SubSectionConstructionAuto} and Anosov-diffeomorphisms in subsection \ref{SubsectionAnosov}), Evgeny Khukhro (for his very valuable feedback on an earlier draft of this paper), Gerry Myerson (for providing references to the relevant results about reduced resultants that informed subsection \ref{SubsecExsPhiPsi}), John Thompson (for allowing the re-use of his proof in subsection \ref{SubsecMainThA}), and Bertram Wehrfritz. The author also thanks Harald Schwab (for his generous help with the project) and the ``Geometry and Analysis on Groups'' seminar at the University of Vienna. 

\bibliographystyle{plain}
\bibliography{MySpecialBibliography}

\end{document}